\documentclass[english,oneside,a4paper,10pt]{article}
\usepackage{amssymb}
\usepackage{latexsym}
\usepackage[english]{babel}
\usepackage{amsmath}
\usepackage{amsfonts}
\usepackage{amsbsy}
\usepackage[mathscr]{eucal}
\usepackage[latin1]{inputenc}
\usepackage{verbatim}
\usepackage{mathrsfs}
\usepackage{graphicx}
\usepackage{float}
\usepackage{bbm}
\usepackage{lscape}
\usepackage{mathdots}
\usepackage{empheq}
\usepackage{color}
\usepackage{array}

\providecommand{\abs}[1]{\left\lvert#1 \right\rvert}

\newtheorem{thr}{Theorem}
\newtheorem{cor}{Corollary}
\newtheorem{lem}{Lemma}

\newenvironment{rem}
{\begin{trivlist}\item[\hskip%
\labelsep{{\it \noindent Remark}}]}{\hfill
\end{trivlist}}

\newenvironment{proof}
{\begin{trivlist}\item[\hskip%
\labelsep{\it \noindent Proof.}]}{\hfill $\square$ \rm
\end{trivlist}}

\numberwithin{equation}{section}

\allowdisplaybreaks

\begin{document}
\begin{center}
{\huge {\bf Spectral properties of anti-heptadiagonal } \\
{\bf persymmetric Hankel matrices}} \\
\vspace{1.0cm}
{\Large Jo\~{a}o Lita da Silva\footnote{\textit{E-mail address:} \texttt{jfls@fct.unl.pt}; \texttt{joao.lita@gmail.com}}} \\
\vspace{0.1cm}
\textit{Department of Mathematics and GeoBioTec
\\ Faculty of Sciences and Technology \\
NOVA University of Lisbon \\ Quinta da Torre, 2829-516 Caparica,
Portugal}
\end{center}

\bigskip

\bigskip

\bigskip

\begin{abstract}
In this paper we express the eigenvalues of anti-heptadiagonal persymmetric Hankel matrices as the zeros of explicit polynomials giving also a representation of its eigenvectors. We present also an expression depending on localizable parameters to compute its integer powers. In particular, an explicit formula not depending on any unknown parameter for the inverse of anti-heptadiagonal persymmetric Hankel matrices is provided.
\end{abstract}

\bigskip


\bigskip


\bigskip

\section{Introduction}\label{sec:1}

\indent

The importance of Hankel matrices in computational mathematics and engineering is well-known. As a matter of fact, these type of matrices have not only a varied and numerous relations with polynomial computations (see, for instance, \cite{Bultheel97}) but also applications in engineering problems of system and control theory (see \cite{Datta88}, \cite{Kailath80} or \cite{Olshevsky01} and the references therein). Recently, several authors have studied particular cases of these matrices in order to derive explicit expressions for its powers (see \cite{Gutierrez11}, \cite{Lita16}, \cite{Rimas13a}, \cite{Rimas13b}, \cite{Wu10}, \cite{Yin08} among others).

The aim of this paper is to explore spectral properties of general anti-heptadiagonal persymmetric Hankel matrices, namely locating its eigenvalues and getting an explicit representation of its eigenvectors. Additionally, it is our purpose announce formulae for the computation of its integer powers using, thereunto, a diagonalization for the referred matrices. Particularly, an expression free of any unknown parameter to calculate the inverse of any anti-heptadiagonal persymmetric Hankel matrix (assuming its nonsingularity) is made available. The essential ingredient in the present approach is a suitable decomposition for these sort of matrices, obtained at the expense of the so-called \emph{bordering technique} originally settled in the eighties for band symmetric Toeplitz matrices. Thereby, the statements announced here besides being a small contribution for the state of art on spectral features of anti-banded matrices, show that techniques developed for band symmetric Toeplitz matrices can be also employed in persymmetric Hankel matrices. We emphasize that, in general, the acquaintance of the eigenvalues and eigenvectors of a symmetric Toeplitz matrix does not allow to know which are exactly the eigenvalues and eigenvectors of the corresponding persymmetric Hankel matrix (obtained by multiplication of the exchange matrix), despite multiple symmetries lead to some relationship between eigenpars of these two classes of matrices.

An $n \times n$ matrix $\mathbf{B}_{n}$ is said to be \emph{banded} when all its nonzero elements are confined within a band formed by diagonals parallel to the main diagonal, i.e. $\left[\mathbf{B}_{n} \right]_{k,\ell} = 0$ when $\abs{k - \ell} > h$, and $\left[\mathbf{B}_{n} \right]_{k,k-h} \neq 0$ or $\left[\mathbf{B}_{n} \right]_{k,k+h} \neq 0$ for at least one value of $k$, where $h$ is the \emph{half-bandwidth} and $2h+1$ is the \emph{bandwidth} (see \cite{Pissanetsky84}, page $13$). An $n \times n$ matrix $\mathbf{A}_{n}$ is said to be \emph{anti-banded} if $\mathbf{J}_{n} \mathbf{A}_{n}$ is a banded matrix where $\mathbf{J}_{n}$ is the $n \times n$ exchange matrix (that is, the matrix such that $\left[\mathbf{J}_{n} \right]_{k,n-k+1} = 1$ for each $1 \leqslant k \leqslant n$ with the remaining entries zero). Particularly, an $n \times n$ matrix which has the form
\begin{equation*}
\left[
\begin{array}{ccccccccccc}
  0 & \ldots & \ldots & \ldots & \ldots & \ldots & 0 & \alpha_{n} & \beta_{n} & \gamma_{n} & d_{n} \\
  \vdots & & & & & \iddots & \alpha_{n-1} & \beta_{n-1} & \gamma_{n-1} & d_{n-1} & c_{n-1} \\
  \vdots & & & & \iddots & \iddots & \beta_{n-2} & \gamma_{n-2} & d_{n-2} & c_{n-2} & b_{n-2} \\
  \vdots & & & \iddots & \iddots & \iddots & \gamma_{n-3} & d_{n-3} & c_{n-3} & b_{n-3} & a_{n-3} \\
  \vdots & & \iddots & \iddots & \iddots & \iddots & d_{n-4} & c_{n-4} & b_{n-4} & a_{n-4} & 0 \\
  \vdots & \iddots & \iddots & \iddots & \iddots & \iddots & \iddots & \iddots & \iddots & \iddots & \vdots \\
  0 & \alpha_{5} & \beta_{5} & \gamma_{5} & d_{5} & \iddots & \iddots & \iddots & \iddots & & \vdots \\
  \alpha_{4} & \beta_{4} & \gamma_{4} & d_{4} & c_{4} & \iddots & \iddots & \iddots & & & \vdots \\
  \beta_{3} & \gamma_{3} & d_{3} & c_{3} & b_{3} & \iddots & \iddots & & & & \vdots \\
  \gamma_{2} & d_{2} & c_{2} & b_{2} & a_{2} & \iddots & & & & & \vdots \\
  d_{1} & c_{1} & b_{1} & a_{1} & 0 & \ldots & \ldots & \ldots & \ldots & \ldots & 0
\end{array}
\right]
\end{equation*}
is called \emph{anti-heptadiagonal}.

Throughout, we shall consider the following $n \times n$ anti-heptadiagonal persymmetric Hankel matrix
\begin{equation}\label{eq:1.1}
\mathbf{H}_{n} = \left[
\begin{array}{ccccccccccc}
  0 & \ldots & \ldots & \ldots & \ldots & \ldots & 0 & a & b & c & d \\
  \vdots & & & & & \iddots & a & b & c & d & c \\
  \vdots & & & & \iddots & \iddots & b & c & d & c & b \\
  \vdots & & & \iddots & \iddots & \iddots & c & d & c & b & a \\
  \vdots & & \iddots & \iddots & \iddots & \iddots & d & c & b & a & 0 \\
  \vdots & \iddots & \iddots & \iddots & \iddots & \iddots & \iddots & \iddots & \iddots & \iddots & \vdots \\
  0 & a & b & c & d & \iddots & \iddots & \iddots & \iddots & & \vdots \\
  a & b & c & d & c & \iddots & \iddots & \iddots & & & \vdots \\
  b & c & d & c & b & \iddots & \iddots & & & & \vdots \\
  c & d & c & b & a & \iddots & & & & & \vdots \\
  d & c & b & a & 0 & \ldots & \ldots & \ldots & \ldots & \ldots & 0
\end{array}
\right].
\end{equation}

\section{Auxiliary tools}\label{sec:2}

\indent

Consider the class of complex matrices defined by
\begin{equation*}
\EuScript{A}_{n} := \left\{\mathbf{A}_{n} \in \EuScript{M}_{n}(\mathbb{C}): \left[\mathbf{A}_{n} \right]_{k-1,\ell} + \left[\mathbf{A}_{n} \right]_{k+1,\ell} = \left[\mathbf{A}_{n} \right]_{k,\ell-1} + \left[\mathbf{A}_{n} \right]_{k,\ell+1}, \, k,\ell = 1,2,\ldots,n \right\}
\end{equation*}
where it is assumed $\left[\mathbf{A}_{n} \right]_{n+1,\ell} = \left[\mathbf{A}_{n} \right]_{k,n+1} = \left[\mathbf{A}_{n} \right]_{0,\ell} = \left[\mathbf{A}_{n} \right]_{k,0} = 0$ for all $k,\ell$.
We begin by retrieve a pair of results stated in \cite{Bini83} (see Propositions 2.1 and 2.2), announced here for complex matrices. The proofs do not suffer any changes from the original ones, whence we will omit the details.


\begin{lem}\label{lem1}
\noindent \newline \textnormal{(a)} The class $\EuScript{A}_{n}$ is the algebra generated over $\mathbb{C}$ by the $n \times n$ matrix
\begin{equation}\label{eq:2.0}
\boldsymbol{\Omega}_{n} = \left[
\begin{array}{ccccccc}
  0 & 1 & 0 & \ldots & \ldots & \ldots & 0 \\
  1 & 0 & 1 & \ddots & & & \vdots \\
  0 & 1 & 0 & \ddots & \ddots & & \vdots \\
  \vdots & \ddots & \ddots & \ddots & \ddots & \ddots & \vdots \\
  \vdots & & \ddots & \ddots & 0 & 1 & 0 \\
  \vdots & & & \ddots & 1 & 0 & 1 \\
  0 & \ldots & \ldots & \ldots & 0 & 1 & 0
\end{array}
\right].
\end{equation}

\medskip

\noindent \textnormal{(b)} If $\mathbf{A}_{n} \in \EuScript{A}_{n}$ and $\mathbf{a}^{\top}$ is its first row then
\begin{equation*}
\mathbf{A}_{n} = \sum_{k=0}^{n-1} \omega_{k+1} \boldsymbol{\Omega}_{n}^{k}
\end{equation*}
where $\boldsymbol{\Omega}_{n}$ is the $n \times n$ matrix \eqref{eq:2.0} and $\boldsymbol{\omega}^{\top} = (\omega_{1},\omega_{2},\ldots,\omega_{n})$ is the solution of the upper triangular system $\mathbf{T}_{n} \boldsymbol{\omega} = \mathbf{a}$, with $\left[\mathbf{T}_{n} \right]_{0,0} := 1$, $\left[\mathbf{T}_{n} \right]_{0,\ell} := 0$ for all $1 \leqslant \ell \leqslant n$,
\begin{equation*}
\left[\mathbf{T}_{n} \right]_{k,\ell} := \left\{
\begin{array}{l}
  \left[\mathbf{T}_{n} \right]_{k-1,\ell-1} + \left[\mathbf{T}_{n} \right]_{k+1,\ell-1}, \quad 1 \leqslant k \leqslant \ell \leqslant n \\
  0, \quad \textit{otherwise}
\end{array}
\right..
\end{equation*}
\end{lem}

Given a positive integer $n$, $\mathbf{S}_{n+2}$ will denote the $(n+2) \times (n+2)$ symmetric, involutory and orthogonal matrix defined by
\begin{equation}\label{eq:2.1}
\left[\mathbf{S}_{n+2} \right]_{k,\ell} := \sqrt{\frac{2}{n+3}} \sin \left(\frac{k \ell \pi}{n+3} \right).
\end{equation}
Our second auxiliary result makes useful Lemma~\ref{lem1} above and provides us an orthogonal diagonalization for the following $(n+2) \times (n+2)$ complex matrix
\begin{equation}\label{eq:2.2}
\mathbf{\widehat{H}}_{n+2} = \left[
\begin{array}{ccccccccccc}
  0 & \ldots & \ldots & \ldots & \ldots & \ldots & 0 & a & b & c - a & d - b \\
  \vdots & & & & & \iddots & a & b & c & d & c - a \\
  \vdots & & & & \iddots & \iddots & b & c & d & c & b \\
  \vdots & & & \iddots & \iddots & \iddots & c & d & c & b & a \\
  \vdots & & \iddots & \iddots & \iddots & \iddots & d & c & b & a & 0 \\
  \vdots & \iddots & \iddots & \iddots & \iddots & \iddots & \iddots & \iddots & \iddots & \iddots & \vdots \\
  0 & a & b & c & d & \iddots & \iddots & \iddots & \iddots & & \vdots \\
  a & b & c & d & c & \iddots & \iddots & \iddots & & & \vdots \\
  b & c & d & c & b & \iddots & \iddots & & & & \vdots \\
  c - a & d & c & b & a & \iddots & & & & & \vdots \\
  d - b & c - a & b & a & 0 & \ldots & \ldots & \ldots & \ldots & \ldots & 0
\end{array}
\right].
\end{equation}


\begin{lem}\label{lem2}
Let $n \in \mathbb{N}$, $a,b,c,d \in \mathbb{C}$ and
\begin{equation}\label{eq:2.3}
\lambda_{k} = -2a \cos \left(\tfrac{n k \pi}{n+3} \right) - 2b \cos \left[\tfrac{(n + 1) k \pi}{n+3} \right] - 2c \cos \left[\tfrac{(n + 2) k \pi}{n+3} \right] - d \cos \left(k \pi \right), \quad k = 1,2,\ldots,n+2.
\end{equation}
If $\mathbf{\widehat{H}}_{n+2}$ is the $(n+2) \times (n+2)$ matrix \eqref{eq:2.2} then
\begin{equation*}
\mathbf{\widehat{H}}_{n+2} = \mathbf{S}_{n+2} \boldsymbol{\Lambda}_{n+2} \mathbf{S}_{n+2}
\end{equation*}
where $\boldsymbol{\Lambda}_{n+2} = \mathrm{diag} \left(\lambda_{1},\lambda_{2},\ldots, \lambda_{n+2} \right)$ and $\mathbf{S}_{n+2}$ is the $(n+2) \times (n+2)$ matrix \eqref{eq:2.1}.
\end{lem}

\begin{proof}
Suppose $n \in \mathbb{N}$ and $a,b,c,d \in \mathbb{C}$. Since $\mathbf{\widehat{H}}_{n+2} \in \EuScript{A}_{n+2}$ and its first row is
\begin{equation*}
\mathbf{a}^{\top} = (0,\ldots,0,a,b,c - a,d - b)
\end{equation*}
we have, from Lemma~\ref{lem1},
\begin{equation*}
\mathbf{\widehat{H}}_{n+2} = \left[(d - 2b) \mathbf{I}_{n+2} + (c - 3a) \boldsymbol{\Omega}_{n+2} + b \, \boldsymbol{\Omega}_{n+2}^{2} + a \, \boldsymbol{\Omega}_{n+2}^{3} \right] \mathbf{J}_{n+2}
\end{equation*}
where $\mathbf{J}_{n+2}$ is the $(n+2) \times (n+2)$ exchange matrix (that is, the matrix such that $\left[\mathbf{J}_{n+2} \right]_{k,n-k+3} = 1$ for every $1 \leqslant k \leqslant n+2$ with the remaining entries zero). Using the spectral decomposition
\begin{equation*}
\boldsymbol{\Omega}_{n+2} = \sum_{\ell=1}^{n+2} 2 \cos \left(\frac{\ell \pi}{n + 3} \right) \mathbf{s}_{\ell} \, \mathbf{s}_{\ell}^{\top},
\end{equation*}
where
\begin{equation*}
\mathbf{s}_{\ell} = \left[
\begin{array}{c}
  \sqrt{\frac{2}{n+3}} \sin \left(\frac{\ell \pi}{n + 3} \right) \\
  \sqrt{\frac{2}{n+3}} \sin \left(\frac{2\ell \pi}{n + 3} \right) \\
  \vdots \\
  \sqrt{\frac{2}{n+3}} \sin \left[\frac{(n + 1)\ell \pi}{n + 3} \right]
\end{array}
\right]
\end{equation*}
(i.e. the $\ell$th column of $\mathbf{S}_{n+2}$ in \eqref{eq:2.1}) and the decomposition
\begin{equation*}
\mathbf{J}_{n+2} = \mathbf{S}_{n+2} \, \mathrm{diag}\left[(-1)^{2},(-1)^{3},\ldots,(-1)^{n+3} \right] \mathbf{S}_{n+2}
\end{equation*}
(see Lemma 1 of \cite{Gutierrez14}) with $\mathbf{S}_{n+2}$ given by \eqref{eq:2.1}, it follows
\begin{equation*}
\begin{split}
\mathbf{\widehat{H}}_{n+2} &= \left\{\sum_{\ell=1}^{n+2} \left[(d - 2b) + 2 (c - 3a) \cos \left(\frac{\ell \pi}{n + 3} \right) + 4b \cos^{2} \left(\frac{\ell \pi}{n + 3} \right) + 8a \cos^{3} \left(\frac{\ell \pi}{n + 3} \right) \right] \mathbf{s}_{\ell} \, \mathbf{s}_{\ell}^{\top} \right\} \mathbf{J}_{n+2} \\
&= \mathbf{S}_{n+2} \, \mathrm{diag} \left[(-1)^{2} \lambda_{1}, (-1)^{3} \lambda_{2},\ldots,(-1)^{n+3} \lambda_{n+2} \right] \, \mathrm{diag} \left[(-1)^{2},(-1)^{3},\ldots,(-1)^{n+3} \right] \mathbf{S}_{n+2} \\
&= \mathbf{S}_{n+2} \, \mathrm{diag} \left(\lambda_{1},\lambda_{2},\ldots, \lambda_{n+2} \right) \mathbf{S}_{n+2}
\end{split}.
\end{equation*}
The proof is completed.
\end{proof}

The decomposition below plays a central role in the study of the spectral properties of anti-heptadiagonal persymmetric Hankel matrices. The statement is split in two cases of $n \in \mathbb{N}$: $n$ is an even number and $n$ is an odd number. For the case in which $n$ is even, we shall take the same steps as in \cite{Bini83} (see pages $106$ and $107$); the case in which $n$ is odd, this procedure must be improved and we shall follow closely the \emph{bordering technique} of Fasino (see \cite{Fasino88}, pages $303$ and $304$).


\begin{lem}\label{lem3}
Let $n \in \mathbb{N}$, $a,b,c,d \in \mathbb{C}$, $\lambda_{k}$, $k=1,2,\ldots,n+2$ be given by \eqref{eq:2.3} and $\mathbf{H}_{n}$ the $n \times n$ matrix \eqref{eq:1.1}.

\medskip

\noindent \textnormal{(a)} If $n$ is even then
\begin{subequations}
\begin{equation}\label{eq:2.4a}
\mathbf{H}_{n} = \mathbf{R}_{n} \mathbf{P}_{n}^{\top} \left[
\begin{array}{cc}
\mathrm{diag} \left(\lambda_{3},\lambda_{5},\ldots,\lambda_{n+1} \right) + \lambda_{1} \mathbf{u} \mathbf{u}^{\top} & \mathbf{O} \\
\mathbf{O} & \mathrm{diag} \left(\lambda_{2},\lambda_{4},\ldots,\lambda_{n} \right) + \lambda_{n+2} \mathbf{v} \mathbf{v}^{\top}
\end{array}
\right] \mathbf{P}_{n} \mathbf{R}_{n}
\end{equation}
where
\begin{equation}\label{eq:2.4b}
\mathbf{u} = \left[
\begin{array}{c}
  \frac{\sin\left(\frac{3\pi}{n+3} \right)}{\sin \left(\frac{\pi}{n+3} \right)} \\[10pt]
  \frac{\sin\left(\frac{5\pi}{n+3} \right)}{\sin \left(\frac{\pi}{n+3} \right)}  \\
  \vdots \\[2pt]
  \frac{\sin\left[\frac{(n+1)\pi}{n+3} \right]}{\sin \left(\frac{\pi}{n+3} \right)}
\end{array}
\right], \quad
\mathbf{v} = \left[
\begin{array}{c}
  \frac{\sin\left(\frac{2\pi}{n+3} \right)}{\sin \left(\frac{\pi}{n+3} \right)} \\[10pt]
  \frac{\sin\left(\frac{4\pi}{n+3} \right)}{\sin \left(\frac{\pi}{n+3} \right)}  \\
  \vdots \\[2pt]
  \frac{\sin\left(\frac{n\pi}{n+3} \right)}{\sin \left(\frac{\pi}{n+3} \right)}
\end{array}
\right],
\end{equation}
$\mathbf{R}_{n}$ is the $n \times n$ matrix given by
\begin{equation}\label{eq:2.4c}
\left[\mathbf{R}_{n} \right]_{k,\ell} = \sqrt{\frac{2}{n + 3}} \sin \left[\frac{(k + 1)(\ell + 1)\pi}{n + 3} \right]
\end{equation}
and $\mathbf{P}_{n}$ is the $n \times n$ permutation matrix defined by
\begin{equation}\label{eq:2.4d}
\left[\mathbf{P}_{n} \right]_{k,\ell} = \left\{
\begin{array}{l}
1 \; \; \textit{if $\ell = 2k - n - 1$ or $\ell = 2k$} \\[5pt]
0 \; \; \textit{otherwise}
\end{array}
\right..
\end{equation}
\end{subequations}

\medskip

\noindent \textnormal{(b)} If $n$ is odd then
\begin{subequations}
\begin{equation}\label{eq:2.5a}
\mathbf{H}_{n} = \mathbf{R}_{n} \mathbf{P}_{n} \left[
\begin{array}{cc}
\mathrm{diag} \left(\lambda_{3},\lambda_{5},\ldots,\lambda_{n}, \lambda_{n+2} \right) + \lambda_{1} \mathbf{u} \mathbf{u}^{\top} & \mathbf{O} \\
\mathbf{O} & \mathrm{diag} \left(\lambda_{2},\lambda_{4},\ldots,\lambda_{n-1} \right) + \lambda_{n+1} \mathbf{v} \mathbf{v}^{\top}
\end{array}
\right] \mathbf{P}_{n}^{\top} \mathbf{R}_{n}^{\top}
\end{equation}
where
\begin{equation}\label{eq:2.5b}
\mathbf{u} = \left[
\begin{array}{c}
  \frac{\sin\left(\frac{3\pi}{n+3} \right)}{\sin \left(\frac{\pi}{n+3} \right)} \\[10pt]
  \frac{\sin\left(\frac{5\pi}{n+3} \right)}{\sin \left(\frac{\pi}{n+3} \right)}  \\
  \vdots \\[2pt]
  \frac{\sin\left[\frac{(n+1)\pi}{n+3} \right]}{\sin \left(\frac{\pi}{n+3} \right)} \\[7pt]
  1
\end{array}
\right], \quad
\mathbf{v} = \left[
\begin{array}{c}
  \frac{\sin\left(\frac{2\pi}{n+3} \right)}{\sin \left[\frac{(n+1)\pi}{n+3} \right]} \\[10pt]
  \frac{\sin\left(\frac{4\pi}{n+3} \right)}{\sin \left[\frac{(n+1)\pi}{n+3} \right]}  \\
  \vdots \\[2pt]
  \frac{\sin\left[\frac{(n-1)\pi}{n+3} \right]}{\sin \left[\frac{(n+1)\pi}{n+3} \right]}
\end{array}
\right],
\end{equation}
$\mathbf{R}_{n}$ is the $n \times n$ matrix given by
\begin{equation}\label{eq:2.5c}
\left[\mathbf{R}_{n} \right]_{k,\ell} = \left\{
\begin{array}{c}
  \sqrt{\frac{2}{n + 3}} \sin \left[\frac{(k + 1)(\ell + 1)\pi}{n + 3} \right] \; \; \textit{if $\ell < n$} \\[10pt]
  \sqrt{\frac{2}{n + 3}} \sin \left[\frac{(k + 1)(n + 2)\pi}{n + 3} \right] \; \; \textit{if $\ell = n$}
\end{array}
\right.
\end{equation}
and $\mathbf{P}_{n}$ is the $n \times n$ permutation matrix defined by
\begin{equation}\label{eq:2.5d}
\left[\mathbf{P}_{n} \right]_{k,\ell} = \left\{
\begin{array}{l}
1 \; \; \textit{if $k = 2\ell$ or $k = 2\ell - n - 2$} \\[5pt]
1 \; \; \textit{if $k = n$ and $\ell = (n+1)/2$} \\[5pt]
0 \; \; \textit{otherwise}
\end{array}
\right..
\end{equation}
\end{subequations}
\end{lem}

\begin{proof}
Consider $n \in \mathbb{N}$, $a,b,c,d \in \mathbb{C}$ and $\lambda_{k}$, $k = 1,2,\ldots,n+2$ given by \eqref{eq:2.3}.

\medskip

\noindent (a) Supposing $n$ even and $\mathbf{S}_{n+2}$ given by \eqref{eq:2.1}, we have
\begin{equation*}
\mathbf{S}_{n+2} = \left[
 \setlength{\extrarowheight}{2pt}
\begin{array}{c|c|c}
  \theta & \mathbf{x}^{\top} & \theta \\ \hline
  \mathbf{x} & \mathbf{R}_{n} & \mathbf{y} \\[2pt] \hline
  \theta & \mathbf{y}^{\top} & -\theta
\end{array}
\right]
\end{equation*}
where $\theta := \sqrt{\frac{2}{n + 3}} \sin \left(\frac{\pi}{n + 3} \right)$,
\begin{equation}\label{eq:2.6}
\mathbf{x} = \left[
\begin{array}{c}
  \sqrt{\frac{2}{n + 3}} \sin \left(\frac{2\pi}{n + 3} \right) \\[10pt]
  \sqrt{\frac{2}{n + 3}} \sin \left(\frac{3\pi}{n + 3} \right) \\
  \vdots \\[2pt]
  \sqrt{\frac{2}{n + 3}} \sin \left[\frac{(n+1) \pi}{n + 3} \right]
\end{array}
\right] \quad \textnormal{and} \quad \mathbf{y} = \left[
\begin{array}{c}
  (-1) \sqrt{\frac{2}{n + 3}} \sin \left(\frac{2\pi}{n + 3} \right) \\[10pt]
  (-1)^{2} \sqrt{\frac{2}{n + 3}} \sin \left(\frac{3\pi}{n + 3} \right) \\
  \vdots \\[2pt]
  (-1)^{n} \sqrt{\frac{2}{n + 3}} \sin \left[\frac{(n+1) \pi}{n + 3} \right]
\end{array}
\right].
\end{equation}
From Lemma~\ref{lem2}, we obtain $\mathbf{\widehat{H}}_{n+2} = \mathbf{S}_{n+2} \, \mathrm{diag} \left(\lambda_{1},\lambda_{2},\ldots,\lambda_{n+2} \right) \mathbf{S}_{n+2}$ and
since $\mathbf{H}_{n}$ is obtained by deleting the first and last row and column of the matrix $\mathbf{\widehat{H}}_{n+2}$ in \eqref{eq:2.2}, it follows
\begin{equation*}
\mathbf{H}_{n} = \mathbf{R}_{n} \left[\mathrm{diag}(\lambda_{2}, \lambda_{3}, \ldots, \lambda_{n+1}) + \lambda_{1} \mathbf{R}_{n}^{-1} \mathbf{x} \left(\mathbf{R}_{n}^{-1} \mathbf{x} \right)^{\top} + \lambda_{n+2} \mathbf{R}_{n}^{-1} \mathbf{y} \left(\mathbf{R}_{n}^{-1} \mathbf{y} \right)^{\top} \right] \mathbf{R}_{n}
\end{equation*}
with
\begin{gather}
\mathbf{R}_{n}^{-1} \mathbf{x} = - \frac{\mathbf{x} + \mathbf{y}}{2\theta} \label{eq:2.7}, \\
\mathbf{R}_{n}^{-1} \mathbf{y} = - \frac{\mathbf{x} - \mathbf{y}}{2\theta}. \label{eq:2.8}
\end{gather}
The odd components of \eqref{eq:2.7} and the even components of \eqref{eq:2.8} are zero; moreover, the $k$th even component of \eqref{eq:2.7} and the $k$th odd component of \eqref{eq:2.8} is
\begin{equation*}
-\frac{\sin \left[\frac{(k+1)\pi}{n+3} \right]}{\sin \left(\frac{\pi}{n+3} \right)}.
\end{equation*}
Permuting rows and columns of $\mathrm{diag}(\lambda_{2}, \lambda_{3}, \ldots, \lambda_{n+1}) + \lambda_{1} \mathbf{R}_{n}^{-1} \mathbf{x} \left(\mathbf{R}_{n}^{-1} \mathbf{x} \right)^{\top} + \lambda_{n+2} \mathbf{R}_{n}^{-1} \mathbf{y} \left(\mathbf{R}_{n}^{-1} \mathbf{y} \right)^{\top}$ according to \eqref{eq:2.4d} yields \eqref{eq:2.4a}.

\medskip

\noindent (b) Consider $n$ odd and the $(n+2) \times (n+2)$ orthogonal matrix $\mathbf{\widehat{S}}_{n+2} := \mathbf{S}_{n+2} \mathbf{K}_{n+2}$, where $\mathbf{S}_{n+2}$ is the matrix \eqref{eq:2.1} and $\mathbf{K}_{n+2}$ is the $(n+2) \times (n+2)$ permutation matrix
\begin{equation*}
\mathbf{K}_{n+2} := \left[
\setlength{\extrarowheight}{2pt}
\begin{array}{c|c|c}
  \mathbf{I}_{n} & \mathbf{0} & \mathbf{0} \\[2pt] \hline
  \mathbf{0}^{\top} & 0 & 1 \\[2pt] \hline
  \mathbf{0}^{\top} & 1 & 0
\end{array}
\right].
\end{equation*}
Using Lemma~\ref{lem2}, we have $\mathbf{\widehat{H}}_{n+2} = \mathbf{\widehat{S}}_{n+2} \, \mathrm{diag} \left(\lambda_{1},\lambda_{2},\ldots,\lambda_{n},\lambda_{n+2},\lambda_{n+1} \right) \mathbf{\widehat{S}}_{n+2}^{\top}$. Putting
\begin{equation*}
\mathbf{\widehat{S}}_{n+2} = \left[
\setlength{\extrarowheight}{2pt}
\begin{array}{c|c|c}
  \theta & \mathbf{z}^{\top} & \eta \\ \hline
  \mathbf{x} & \mathbf{R}_{n} & \mathbf{y} \\[2pt] \hline
  \theta & \mathbf{w}^{\top} & -\eta
\end{array}
\right]
\end{equation*}
where $\theta := \sqrt{\frac{2}{n + 3}} \sin \left(\frac{\pi}{n + 3} \right)$, $\eta := \sqrt{\frac{2}{n + 3}} \sin \left[\frac{(n+1)\pi}{n + 3} \right]$,

\begin{equation} \label{eq:2.9}
\begin{split}
& \mathbf{x} = \left[
\begin{array}{c}
  \sqrt{\frac{2}{n + 3}} \sin \left(\frac{2\pi}{n + 3} \right) \\[10pt]
  \sqrt{\frac{2}{n + 3}} \sin \left(\frac{3\pi}{n + 3} \right) \\
  \vdots \\[2pt]
  \sqrt{\frac{2}{n + 3}} \sin \left[\frac{(n+1) \pi}{n + 3} \right]
\end{array}
\right], \quad \mathbf{y} = \left[
\begin{array}{c}
  (-1) \sqrt{\frac{2}{n + 3}} \sin \left(\frac{4\pi}{n + 3} \right) \\[10pt]
  (-1)^{2} \sqrt{\frac{2}{n + 3}} \sin \left(\frac{6\pi}{n + 3} \right) \\
  \vdots \\[2pt]
  (-1)^{n} \sqrt{\frac{2}{n + 3}} \sin \left[\frac{2(n+1) \pi}{n + 3} \right]
\end{array}
\right] \\
& \mathbf{z} = \left[
\begin{array}{c}
  \sqrt{\frac{2}{n + 3}} \sin \left(\frac{2\pi}{n + 3} \right) \\[10pt]
  \sqrt{\frac{2}{n + 3}} \sin \left(\frac{3\pi}{n + 3} \right) \\
  \vdots \\[2pt]
  \sqrt{\frac{2}{n + 3}} \sin \left(\frac{n\pi}{n + 3} \right) \\[10pt]
  \sqrt{\frac{2}{n + 3}} \sin \left[\frac{(n+2) \pi}{n + 3} \right]
\end{array}
\right], \quad \mathbf{w} = \left[
\begin{array}{c}
  (-1) \sqrt{\frac{2}{n + 3}} \sin \left(\frac{2\pi}{n + 3} \right) \\[10pt]
  (-1)^{2} \sqrt{\frac{2}{n + 3}} \sin \left(\frac{3\pi}{n + 3} \right) \\
  \vdots \\[2pt]
  (-1)^{n-1} \sqrt{\frac{2}{n + 3}} \sin \left(\frac{n\pi}{n + 3} \right) \\[10pt]
   \sqrt{\frac{2}{n + 3}} \sin \left[\frac{(n+2) \pi}{n + 3} \right]
\end{array}
\right]
\end{split}
\end{equation}
and $\mathbf{R}_{n}$ given by \eqref{eq:2.5c}, we get
\begin{equation*}
\mathbf{H}_{n} = \mathbf{R}_{n} \left[\mathrm{diag}(\lambda_{2}, \lambda_{3}, \ldots, \lambda_{n},\lambda_{n+2}) + \lambda_{1} \mathbf{R}_{n}^{-1} \mathbf{x} \left(\mathbf{R}_{n}^{-1} \mathbf{x} \right)^{\top} + \lambda_{n+1} \mathbf{R}_{n}^{-1} \mathbf{y} \left(\mathbf{R}_{n}^{-1} \mathbf{y} \right)^{\top} \right] \mathbf{R}_{n}^{\top}
\end{equation*}
with
\begin{gather}
\mathbf{R}_{n}^{-1} \mathbf{x} = - \frac{\mathbf{z} + \mathbf{w}}{2\theta}, \label{eq:2.10} \\
\mathbf{R}_{n}^{-1} \mathbf{y} = - \frac{\mathbf{z} - \mathbf{w}}{2\eta} \label{eq:2.11}
\end{gather}
since $\mathbf{H}_{n}$ is obtained by deleting the first and last row and column of the matrix $\mathbf{\widehat{H}}_{n+2}$. The first $n-1$ odd components of \eqref{eq:2.10} are zero; additionally, the even components of \eqref{eq:2.11} are zero as well as its last component. In the first $n-1$ components of \eqref{eq:2.10}, the $k$th even component is
\begin{equation*}
-\frac{\sin \left[\frac{(k+1)\pi}{n+3} \right]}{\sin \left(\frac{\pi}{n+3} \right)},
\end{equation*}
and the $k$th odd component of \eqref{eq:2.11} is
\begin{equation*}
-\frac{\sin \left[\frac{(k+1)\pi}{n+3} \right]}{\sin \left[\frac{(n+1)\pi}{n + 3} \right]}.
\end{equation*}
Permuting rows and columns of
\begin{equation*}
\mathrm{diag}(\lambda_{2}, \lambda_{3}, \ldots, \lambda_{n},\lambda_{n+2}) + \lambda_{1} \mathbf{R}_{n}^{-1} \mathbf{x} \left(\mathbf{R}_{n}^{-1} \mathbf{x} \right)^{\top} + \lambda_{n+1} \mathbf{R}_{n}^{-1} \mathbf{y} \left(\mathbf{R}_{n}^{-1} \mathbf{y} \right)^{\top}
\end{equation*}
according to \eqref{eq:2.5d} we obtain \eqref{eq:2.5a} establishing the thesis.
\end{proof}

\begin{rem}
Let us observe that for $n$ odd the condition $\left[\mathbf{\widehat{S}}_{n+2} \right]_{n+2,\ell} = (-1)^{\ell + 1} \left[\mathbf{\widehat{S}}_{n+2} \right]_{1,\ell}$, $\ell=1,2,\ldots,n+2$ is not satisfied (unlike the case in which $n$ is even and we have $\left[\mathbf{S}_{n+2} \right]_{n+2,\ell} = (-1)^{\ell + 1} \left[\mathbf{S}_{n+2} \right]_{1,\ell}$, $\ell=1,2,\ldots,n+2$). This foreshadowed that the bordering technique can be founded taking account only the four corners of the orthogonal matrix (in order to get $\det(\mathbf{R}_{n}) \neq 0$) and discarding any relation between the remaining entries of its first and last rows.
\end{rem}

\section{Spectral properties of $\mathbf{H}_{n}$}\label{sec:3}

\subsection{Eigenvalue localization}

\indent

The decomposition presented in Lemma~\ref{lem3} allows us to express the eigenvalues of $\mathbf{H}_{n}$ as the zeros of explicit polynomials.

\begin{thr}\label{thr1}
Let $n \in \mathbb{N}$, $a,b,c,d \in \mathbb{C}$, $\lambda_{k}$, $k=1,2,\ldots,n+2$ be given by \eqref{eq:2.3} and $\mathbf{H}_{n}$ the $n \times n$ matrix \eqref{eq:1.1}.

\medskip

\begin{subequations}
\noindent \textnormal{(a)} If $n$ is even then the eigenvalues of $\mathbf{H}_{n}$ are precisely the zeros of
\begin{equation}\label{eq:3.1a}
f(t) = \prod_{\substack{j=1}}^{\frac{n}{2}} (t - \lambda_{2j+1}) + (t - \lambda_{1}) \sum_{j=1}^{\frac{n}{2}} \frac{\sin^{2} \left[\frac{(2j + 1)\pi}{n + 3} \right]}{\sin^{2} \left(\frac{\pi}{n + 3} \right)} \prod_{\substack{\substack{m=1 \\ m \neq j}}}^{\frac{n}{2}} (t - \lambda_{2m+1})
\end{equation}
and
\begin{equation}\label{eq:3.1b}
g(t) = \prod_{\substack{j=1}}^{\frac{n}{2}} (t - \lambda_{2j}) + (t - \lambda_{n+2}) \sum_{j=1}^{\frac{n}{2}} \frac{\sin^{2} \left(\frac{2j\pi}{n + 3} \right)}{\sin^{2} \left(\frac{\pi}{n + 3} \right)} \prod_{\substack{m=1 \\ m \neq j}}^{\frac{n}{2}} (t - \lambda_{2m}).
\end{equation}
\end{subequations}

\medskip

\begin{subequations}
\noindent \textnormal{(b)} If $n$ is odd then the eigenvalues of $\mathbf{H}_{n}$ are precisely the zeros of
\begin{equation}\label{eq:3.2a}
f(t) = \prod_{j=1}^{\frac{n+1}{2}} (t - \lambda_{2j+1}) + (t - \lambda_{1}) \sum_{j=1}^{\frac{n+1}{2}} \frac{\sin^{2} \left[\frac{(2j + 1)\pi}{n + 3} \right]}{\sin^{2} \left(\frac{\pi}{n + 3} \right)} \prod_{\substack{m=1 \\ m \neq j}}^{\frac{n+1}{2}}(t - \lambda_{2m+1})
\end{equation}
and
\begin{equation}\label{eq:3.2b}
g(t) = \prod_{j=1}^{\frac{n-1}{2}}(t - \lambda_{2j}) + (t - \lambda_{n+1}) \sum_{j=1}^{\frac{n-1}{2}} \frac{\sin^{2} \left(\frac{2j\pi}{n + 3} \right)}{\sin^{2} \left[\frac{(n+1)\pi}{n + 3} \right]} \prod_{\substack{m=1 \\ m \neq j}}^{\frac{n-1}{2}} (t - \lambda_{2m}).
\end{equation}
\end{subequations}
\end{thr}

\begin{proof}
Let $n \in \mathbb{N}$, $a,b,c,d \in \mathbb{C}$ and $\lambda_{k}$, $k=1,2,\ldots,n+2$ be given by \eqref{eq:2.3}.

\medskip

\noindent (a) Consider $n$ even and $\mathbf{R}_{n}$ the matrix defined by \eqref{eq:2.4c}. We have
\begin{equation*}
\mathbf{R}_{n}^{2} = \mathbf{I}_{n} - \mathbf{x} \mathbf{x}^{\top} - \mathbf{y} \mathbf{y}^{\top}
\end{equation*}
where $\mathbf{x}$, $\mathbf{y}$ are given by \eqref{eq:2.6} and $\det(\mathbf{R}_{n}^{2}) = \frac{16}{(n+3)^{2}} \sin^{4} \left(\frac{\pi}{n + 3} \right) \neq 0$ (see \cite{Bini83}, page $106$), so that
\begin{equation*}
\mathbf{R}_{n}^{-2} = \left(\mathbf{I}_{n} - \mathbf{x} \mathbf{x}^{\top} - \mathbf{y} \mathbf{y}^{\top} \right)^{-1}.
\end{equation*}
Setting
\begin{gather*}
\boldsymbol{\Sigma}_{\frac{n}{2}} := \mathrm{diag} \left(\lambda_{3},\lambda_{5},\ldots,\lambda_{n+1} \right) + \lambda_{1} \mathbf{u} \mathbf{u}^{\top} \\
\boldsymbol{\Gamma}_{\frac{n}{2}} := \mathrm{diag} \left(\lambda_{2},\lambda_{4},\ldots,\lambda_{n} \right) + \lambda_{n+2} \mathbf{v} \mathbf{v}^{\top}
\end{gather*}
with $\mathbf{u}$, $\mathbf{v}$ given by \eqref{eq:2.4b}, we obtain
\begin{align*}
\det & (t \mathbf{I}_{n} - \mathbf{H}_{n}) = \\
&= \left[\det(\mathbf{R}_{n}) \right]^{2} \det \left(t \mathbf{R}_{n}^{-2} - \mathbf{P}_{n}^{\top} \left[
\begin{array}{cc}
  \boldsymbol{\Sigma}_{\frac{n}{2}} & \mathbf{O} \\
  \mathbf{O} & \boldsymbol{\Gamma}_{\frac{n}{2}}
\end{array}
\right] \mathbf{P}_{n} \right) \\
&= \det \left(\mathbf{R}_{n}^{2} \right) \det \left(t \mathbf{P}_{n} \mathbf{R}_{n}^{-2} \mathbf{P}_{n}^{\top} - \left[
\begin{array}{cc}
  \boldsymbol{\Sigma}_{\frac{n}{2}} & \mathbf{O} \\
  \mathbf{O} & \boldsymbol{\Gamma}_{\frac{n}{2}}
\end{array}
\right] \right) \\
&= \det \left(\mathbf{R}_{n}^{2} \right) \det \left(t \left[\mathbf{P}_{n} \left(\mathbf{I}_{n} - \mathbf{x} \mathbf{x}^{\top} - \mathbf{y} \mathbf{y}^{\top} \right) \mathbf{P}_{n}^{\top} \right]^{-1} - \left[
\begin{array}{cc}
  \boldsymbol{\Sigma}_{\frac{n}{2}} & \mathbf{O} \\
  \mathbf{O} & \boldsymbol{\Gamma}_{\frac{n}{2}}
\end{array}
\right] \right) \\
&= \det \left(\mathbf{R}_{n}^{2} \right) \det \left(t \left[
\begin{array}{cc}
  \mathbf{I}_{\frac{n}{2}} - \frac{4 \sin^{2} \left(\frac{\pi}{n+3} \right)}{n+3} \mathbf{u} \mathbf{u}^{\top} & \mathbf{O} \\
  \mathbf{O} & \mathbf{I}_{\frac{n}{2}} - \frac{4 \sin^{2} \left(\frac{\pi}{n+3} \right)}{n+3} \mathbf{v} \mathbf{v}^{\top}
\end{array}
\right]^{-1} - \left[
\begin{array}{cc}
  \boldsymbol{\Sigma}_{\frac{n}{2}} & \mathbf{O} \\
  \mathbf{O} & \boldsymbol{\Gamma}_{\frac{n}{2}}
\end{array}
\right] \right) \\
& = \det \left(\mathbf{R}_{n}^{2} \right) \det \left(t \left[
\begin{array}{cc}
  \mathbf{I}_{\frac{n}{2}} - \frac{1}{\mathbf{u}^{\top} \mathbf{u} - \frac{n+3}{4 \sin^{2} \left(\frac{\pi}{n+3} \right)}} \mathbf{u} \mathbf{u}^{\top} & \mathbf{O} \\
  \mathbf{O} & \mathbf{I}_{\frac{n}{2}} - \frac{1}{\mathbf{v}^{\top} \mathbf{v} - \frac{n+3}{4 \sin^{2} \left(\frac{\pi}{n+3} \right)}} \mathbf{v} \mathbf{v}^{\top}
\end{array}
\right] - \left[
\begin{array}{cc}
  \boldsymbol{\Sigma}_{\frac{n}{2}} & \mathbf{O} \\
  \mathbf{O} & \boldsymbol{\Gamma}_{\frac{n}{2}}
\end{array}
\right] \right) \\
&= \det \left(\mathbf{R}_{n}^{2} \right) \det \left\{t \mathbf{I}_{\frac{n}{2}} - \mathrm{diag} \left(\lambda_{3}, \lambda_{5}, \ldots, \lambda_{n+1} \right) - \left[\lambda_{1} + \frac{t}{\mathbf{u}^{\top} \mathbf{u} - \frac{n+3}{4 \sin^{2} \left(\frac{\pi}{n+3} \right)}} \right] \mathbf{u} \mathbf{u}^{\top} \right\} \cdot \\
& \hspace*{3.7cm} \det \left\{t \mathbf{I}_{\frac{n}{2}} - \mathrm{diag} \left(\lambda_{2}, \lambda_{4}, \ldots, \lambda_{n} \right) - \left[\lambda_{n+2} + \frac{t}{\mathbf{v}^{\top} \mathbf{v} - \frac{n+3}{4 \sin^{2} \left(\frac{\pi}{n+3} \right)}} \right] \mathbf{v} \mathbf{v}^{\top} \right\}
\end{align*}
(see \cite{Harville97}, page $88$ and \cite{Horn13}, page $19$) where $\mathbf{P}_{n}$ the permutation matrix \eqref{eq:2.4d}. Thus, supposing $t \neq \lambda_{2k+1}$ for each $k=1,\ldots,\frac{n}{2}$ and noting that $\mathbf{u}^{\top} \mathbf{u} - \frac{n + 3}{4 \sin^{2} \left(\frac{\pi}{n + 3} \right)} = -1$, we get
\begin{equation}\label{eq:3.3}
\begin{split}
&\det \left\{t \mathbf{I}_{\frac{n}{2}} - \mathrm{diag} \left(\lambda_{3}, \lambda_{5}, \ldots, \lambda_{n+1} \right) - \left[\lambda_{1} + \frac{t}{\mathbf{u}^{\top} \mathbf{u} - \frac{n+3}{4 \sin^{2} \left(\frac{\pi}{n+3} \right)}} \right] \mathbf{u} \mathbf{u}^{\top} \right\} = \\
&\hspace*{2.0cm} = \left\{1 - (\lambda_{1} - t) \sum_{j=1}^{\frac{n}{2}} \frac{\sin^{2} \left[\frac{(2j + 1)\pi}{n + 3} \right]}{(t - \lambda_{2j+1}) \sin^{2} \left(\frac{\pi}{n + 3} \right)} \right\} \prod_{j=1}^{\frac{n}{2}}(t - \lambda_{2j+1}) \\
&\hspace*{2.0cm} = \prod_{j=1}^{\frac{n}{2}}(t - \lambda_{2j+1}) + (t - \lambda_{1}) \sum_{j=1}^{\frac{n}{2}} \frac{\sin^{2} \left[\frac{(2j + 1)\pi}{n + 3} \right]}{\sin^{2} \left(\frac{\pi}{n + 3} \right)} \prod_{\substack{m=1 \\ m \neq j}}^{\frac{n}{2}} (t - \lambda_{2m+1}).
\end{split}
\end{equation}
Since the determinant is a continuous function, relation \eqref{eq:3.3} holds for all $t$. Analogously, we have $\mathbf{v}^{\top} \mathbf{v} - \frac{n + 3}{4 \sin^{2} \left(\frac{\pi}{n + 3} \right)} = -1$ and
\begin{align*}
& \det \left\{t \mathbf{I}_{\frac{n}{2}} - \mathrm{diag} \left(\lambda_{2}, \lambda_{4}, \ldots, \lambda_{n} \right) - \left[\lambda_{n+2} + \frac{t}{\mathbf{v}^{\top} \mathbf{v} - \frac{n+3}{4 \sin^{2} \left(\frac{\pi}{n+3} \right)}} \right] \mathbf{v} \mathbf{v}^{\top} \right\} = \\
&\hspace*{2.0cm} = \left[1 - (\lambda_{n+2} - t) \sum_{j=1}^{\frac{n}{2}} \frac{\sin^{2} \left(\frac{2j\pi}{n + 3} \right)}{(t - \lambda_{2j}) \sin^{2} \left(\frac{\pi}{n + 3} \right)} \right] \prod_{j=1}^{\frac{n}{2}}(t - \lambda_{2j}) \\
& \hspace*{2.0cm} = \prod_{j=1}^{\frac{n}{2}}(t - \lambda_{2j}) + (t - \lambda_{n+2}) \sum_{j=1}^{\frac{n}{2}} \frac{\sin^{2} \left(\frac{2j\pi}{n + 3} \right)}{\sin^{2} \left(\frac{\pi}{n + 3} \right)} \prod_{\substack{m=1 \\ m \neq j}}^{\frac{n}{2}} (t - \lambda_{2m})
\end{align*}
for all $t$, so that the characteristic polynomial of $\mathbf{H}_{n}$ is
\begin{align*}
\det (t \mathbf{I}_{n} - \mathbf{H}_{n}) & = \frac{16 \sin^{4} \left(\frac{\pi}{n + 3} \right)}{(n+3)^{2}} \left\{\prod_{\substack{j=1}}^{\frac{n}{2}} (t - \lambda_{2j+1}) + (t - \lambda_{1}) \sum_{j=1}^{\frac{n}{2}} \frac{\sin^{2} \left[\frac{(2j + 1)\pi}{n + 3} \right]}{\sin^{2} \left(\frac{\pi}{n + 3} \right)}  \cdot \prod_{\substack{\substack{m=1 \\ m \neq j}}}^{\frac{n}{2}} (t - \lambda_{2m+1}) \right\} \cdot \\
&  \\
& \hspace*{3.5cm} \left[\prod_{\substack{j=1}}^{\frac{n}{2}} (t - \lambda_{2j}) + (t - \lambda_{n+2}) \sum_{j=1}^{\frac{n}{2}} \frac{\sin^{2} \left(\frac{2j\pi}{n + 3} \right)}{\sin^{2} \left(\frac{\pi}{n + 3} \right)} \prod_{\substack{m=1 \\ m \neq j}}^{\frac{n}{2}} (t - \lambda_{2m}) \right].
\end{align*}

\medskip

\noindent (b) Supposing $n$ odd and $\mathbf{R}_{n}$ the matrix defined by \eqref{eq:2.5c}, it follows
\begin{equation*}
\mathbf{R}_{n}^{\top} \mathbf{R}_{n} = \mathbf{I}_{n} - \mathbf{z} \mathbf{z}^{\top} - \mathbf{w} \mathbf{w}^{\top}
\end{equation*}
where $\mathbf{z}$ and $\mathbf{w}$ are given in \eqref{eq:2.9}. Hence,
\begin{align*}
\det(\mathbf{R}_{n}^{\top} \mathbf{R}_{n} ) &= \det(\mathbf{I}_{n} - \mathbf{z} \mathbf{z}^{\top} - \mathbf{w} \mathbf{w}^{\top}) \\
&= (1 - \mathbf{z}^{\top} \mathbf{z}) \left[1 - \mathbf{w}^{\top} \mathbf{w} - \frac{(\mathbf{z}^{\top} \mathbf{w})^{2}}{1 - \mathbf{z}^{\top} \mathbf{z}} \right] \\
&= \frac{16}{(n+3)^{2}} \sin^{2} \left(\frac{\pi}{n + 3} \right) \sin^{2} \left[\frac{(n + 1)\pi}{n + 3} \right] \neq 0
\end{align*}
(see \cite{Miller81}, page $70$) and
\begin{equation*}
\left(\mathbf{R}_{n}^{\top} \mathbf{R}_{n} \right)^{-1} = \left(\mathbf{I}_{n} - \mathbf{z} \mathbf{z}^{\top} - \mathbf{w} \mathbf{w}^{\top} \right)^{-1}.
\end{equation*}
Putting
\begin{gather*}
\boldsymbol{\Sigma}_{\frac{n+1}{2}} := \mathrm{diag} \left(\lambda_{3},\lambda_{5},\ldots,\lambda_{n}, \lambda_{n+2} \right) + \lambda_{1} \mathbf{u} \mathbf{u}^{\top} \\
\boldsymbol{\Gamma}_{\frac{n-1}{2}} := \mathrm{diag} \left(\lambda_{2},\lambda_{4},\ldots,\lambda_{n-1} \right) + \lambda_{n+1} \mathbf{v} \mathbf{v}^{\top}
\end{gather*}
with $\mathbf{u}$, $\mathbf{v}$ given by \eqref{eq:2.5b}, we obtain
\begin{align*}
& \det (t \mathbf{I}_{n} - \mathbf{H}_{n}) = \\
& \; = \det(\mathbf{R}_{n}) \det \left(t \mathbf{R}_{n}^{-1} (\mathbf{R}_{n}^{\top})^{-1} - \mathbf{P}_{n} \left[
\begin{array}{cc}
  \boldsymbol{\Sigma}_{\frac{n+1}{2}} & \mathbf{O} \\
  \mathbf{O} & \boldsymbol{\Gamma}_{\frac{n-1}{2}}
\end{array}
\right] \mathbf{P}_{n}^{\top} \right) \det(\mathbf{R}_{n}^{\top}) \\
& \; = \det \left(\mathbf{R}_{n}^{\top} \mathbf{R}_{n}\right) \det \left(t \mathbf{P}_{n}^{\top} (\mathbf{R}_{n}^{\top}\mathbf{R}_{n})^{-1} \mathbf{P}_{n} - \left[
\begin{array}{cc}
  \boldsymbol{\Sigma}_{\frac{n+1}{2}} & \mathbf{O} \\
  \mathbf{O} & \boldsymbol{\Gamma}_{\frac{n-1}{2}}
\end{array}
\right] \right) \\
& \; = \det \left(\mathbf{R}_{n}^{\top} \mathbf{R}_{n}\right) \det \left(t \left[\mathbf{P}_{n}^{\top} \left(\mathbf{I}_{n} - \mathbf{z} \mathbf{z}^{\top} - \mathbf{w} \mathbf{w}^{\top} \right) \mathbf{P}_{n} \right]^{-1} - \left[
\begin{array}{cc}
  \boldsymbol{\Sigma}_{\frac{n+1}{2}} & \mathbf{O} \\
  \mathbf{O} & \boldsymbol{\Gamma}_{\frac{n-1}{2}}
\end{array}
\right] \right) \\
& \; = \det \left(\mathbf{R}_{n}^{\top} \mathbf{R}_{n}\right) \det \left(t \left[
\begin{array}{cc}
  \mathbf{I}_{\frac{n+1}{2}} - \frac{4 \sin^{2} \left(\frac{\pi}{n+3} \right)}{n+3} \mathbf{u} \mathbf{u}^{\top} & \mathbf{O} \\
  \mathbf{O} & \mathbf{I}_{\frac{n-1}{2}} - \frac{4 \sin^{2} \left[\frac{(n+1)\pi}{n+3} \right]}{n+3} \mathbf{v} \mathbf{v}^{\top}
\end{array}
\right]^{-1} - \left[
\begin{array}{cc}
  \boldsymbol{\Sigma}_{\frac{n+1}{2}} & \mathbf{O} \\
  \mathbf{O} & \boldsymbol{\Gamma}_{\frac{n-1}{2}}
\end{array}
\right] \right) \\
& \; = \det \left(\mathbf{R}_{n}^{\top} \mathbf{R}_{n}\right) \cdot \\
& \hspace*{1.2cm} \det \left(t \left[
\begin{array}{cc}
  \mathbf{I}_{\frac{n+1}{2}} - \frac{1}{\mathbf{u}^{\top} \mathbf{u} - \frac{n+3}{4 \sin^{2} \left(\frac{\pi}{n+3} \right)}} \mathbf{u} \mathbf{u}^{\top} & \mathbf{O} \\
  \mathbf{O} & \mathbf{I}_{\frac{n-1}{2}} - \frac{1}{\mathbf{v}^{\top} \mathbf{v} - \frac{n+3}{4 \sin^{2} \left[\frac{(n+1)\pi}{n+3} \right]}} \mathbf{v} \mathbf{v}^{\top}
\end{array}
\right] - \left[
\begin{array}{cc}
  \boldsymbol{\Sigma}_{\frac{n+1}{2}} & \mathbf{O} \\
  \mathbf{O} & \boldsymbol{\Gamma}_{\frac{n-1}{2}}
\end{array}
\right] \right) \\
& \; = \det \left(\mathbf{R}_{n}^{\top} \mathbf{R}_{n}\right) \det \left\{t \mathbf{I}_{\frac{n+1}{2}} - \mathrm{diag} \left(\lambda_{3}, \lambda_{5}, \ldots, \lambda_{n},\lambda_{n+2} \right) - \left[\lambda_{1} + \frac{t}{\mathbf{u}^{\top} \mathbf{u} - \frac{n+3}{4 \sin^{2} \left(\frac{\pi}{n+3} \right)}} \right] \mathbf{u} \mathbf{u}^{\top} \right\} \cdot \\
& \hspace*{3.8cm} \det \left\{t \mathbf{I}_{\frac{n-1}{2}} - \mathrm{diag} \left(\lambda_{2}, \lambda_{4}, \ldots, \lambda_{n-1} \right) - \left[\lambda_{n+1} + \frac{t}{\mathbf{v}^{\top} \mathbf{v} - \frac{n+3}{4 \sin^{2} \left[\frac{(n+1)\pi}{n+3} \right]}} \right] \mathbf{v} \mathbf{v}^{\top} \right\}
\end{align*}
(see \cite{Harville97}, page $88$ and \cite{Horn13}, page $19$) where $\mathbf{P}_{n}$ the permutation matrix \eqref{eq:2.5d}. Thus, admitting $t \neq \lambda_{2k+1}$ for each $k=1,\ldots,\frac{n+1}{2}$, we have
\begin{equation}\label{eq:3.4}
\begin{split}
&\det \left\{t \mathbf{I}_{\frac{n+1}{2}} - \mathrm{diag} \left(\lambda_{3}, \lambda_{5}, \ldots, \lambda_{n},\lambda_{n+2} \right) - \left[\lambda_{1} + \frac{t}{\mathbf{u}^{\top} \mathbf{u} - \frac{n+3}{4 \sin^{2} \left(\frac{\pi}{n+3} \right)}} \right] \mathbf{u} \mathbf{u}^{\top} \right\} = \\
&\qquad = \left\{1 - (\lambda_{1} - t) \left[\sum_{j=1}^{\frac{n-1}{2}} \frac{\sin^{2} \left(\frac{(2j + 1)\pi}{n + 3} \right)}{(t - \lambda_{2j+1}) \sin^{2} \left(\frac{\pi}{n + 3} \right)} + \frac{1}{t - \lambda_{n+2}} \right] \right\} (t - \lambda_{n+2}) \prod_{j=1}^{\frac{n-1}{2}}(t - \lambda_{2j+1}) \\
&\qquad = \left\{1 - (\lambda_{1} - t) \sum_{j=1}^{\frac{n+1}{2}} \frac{\sin^{2} \left[\frac{(2j + 1)\pi}{n + 3} \right]}{(t - \lambda_{2j+1}) \sin^{2} \left(\frac{\pi}{n + 3} \right)} \right\} \prod_{j=1}^{\frac{n+1}{2}}(t - \lambda_{2j+1}) \\
&\qquad = \prod_{j=1}^{\frac{n+1}{2}} (t - \lambda_{2j+1}) + (t - \lambda_{1}) \sum_{j=1}^{\frac{n+1}{2}} \frac{\sin^{2} \left[\frac{(2j + 1)\pi}{n + 3} \right]}{\sin^{2} \left(\frac{\pi}{n + 3} \right)} \prod_{\substack{m=1 \\ m \neq j}}^{\frac{n+1}{2}}(t - \lambda_{2m+1})
\end{split}
\end{equation}
provided that $\mathbf{u}^{\top} \mathbf{u} - \frac{n + 3}{4 \sin^{2} \left(\frac{\pi}{n + 3} \right)} = -1$. Since the determinant is a continuous function, relation \eqref{eq:3.4} is valid for every $t$. Analogously, we have $\mathbf{v}^{\top} \mathbf{v} - \frac{n + 3}{4 \sin^{2} \left[\frac{(n + 1)\pi}{n + 3} \right]} = -1$ and
\begin{align*}
& \det \left\{t \mathbf{I}_{\frac{n-1}{2}} - \mathrm{diag} \left(\lambda_{2}, \lambda_{4}, \ldots, \lambda_{n-1} \right) - \left[\lambda_{n+1} + \frac{t}{\mathbf{v}^{\top} \mathbf{v} - \frac{n+3}{4 \sin^{2} \left(\frac{(n+1)\pi}{n+3} \right)}} \right] \mathbf{v} \mathbf{v}^{\top} \right\} = \\
& \hspace*{5.0cm} = \prod_{j=1}^{\frac{n-1}{2}}(t - \lambda_{2j}) + (t - \lambda_{n+1}) \sum_{j=1}^{\frac{n-1}{2}} \frac{\sin^{2} \left(\frac{2j\pi}{n + 3} \right)}{\sin^{2} \left[\frac{(n+1)\pi}{n + 3} \right]} \prod_{\substack{m=1 \\ m \neq j}}^{\frac{n-1}{2}} (t - \lambda_{2m})
\end{align*}
for all $t$, whence the characteristic polynomial of $\mathbf{H}_{n}$ is
\begin{align*}
\det (t \mathbf{I}_{n} - \mathbf{H}_{n}) & = \frac{16}{(n+3)^{2}} \sin^{2} \left(\frac{\pi}{n + 3} \right) \sin^{2} \left[\frac{(n + 1)\pi}{n + 3} \right] \Bigg\{(t - \lambda_{n+2}) \prod_{j=1}^{\frac{n-1}{2}}(t - \lambda_{2j+1}) + \Bigg. \\
& \hspace*{1.65cm} \Bigg. (t - \lambda_{1})\Bigg[\sum_{j=1}^{\frac{n-1}{2}} \frac{(\lambda_{n+2} - t) \sin^{2} \left(\frac{(2j + 1)\pi}{n + 3} \right)}{\sin^{2} \left(\frac{\pi}{n + 3} \right)} \prod_{\substack{m=1 \\ m \neq j}}^{\frac{n-1}{2}}(t - \lambda_{2j+1}) + \prod_{j=1}^{\frac{n-1}{2}}(t - \lambda_{2j+1}) \Bigg] \Bigg\} \cdot \\
& \hspace*{3.7cm} \left\{\prod_{j=1}^{\frac{n-1}{2}}(t - \lambda_{2j}) - (\lambda_{n+1} - t) \sum_{j=1}^{\frac{n-1}{2}} \frac{\sin^{2} \left(\frac{2j\pi}{n + 3} \right)}{\sin^{2} \left[\frac{(n+1)\pi}{n + 3} \right]} \prod_{\substack{m=1 \\ m \neq j}}^{\frac{n-1}{2}} (t - \lambda_{2m}) \right\}.
\end{align*}
The proof is completed.
\end{proof}

The previous theorem leads us to the following separation result for the eigenvalues of real matrices $\mathbf{H}_{n}$.

\begin{cor}\label{cor1}
Let $n \in \mathbb{N}$, $a,b,c,d \in \mathbb{R}$, $\lambda_{k}$, $k=1,2,\ldots,n+2$ be given by \eqref{eq:2.3} and $\mathbf{H}_{n}$ the $n \times n$ matrix \eqref{eq:1.1}.

\medskip

\noindent \textnormal{(a)} If $n$ is even and:

\begin{subequations}
\begin{itemize}
\item[\textnormal{i.}] $\lambda_{\tau(1)} \leqslant \lambda_{\tau(3)} \leqslant \ldots \leqslant \lambda_{\tau(n+1)}$ are arranged in nondecreasing order by some bijection $\tau$ then
\begin{equation}\label{eq:3.5a}
\lambda_{\tau(2k-1)} \leqslant \mu_{k} \leqslant \lambda_{\tau(2k+1)}
\end{equation}
for any $k=1,\ldots,\frac{n}{2}$, where $\mu_{1},\ldots,\mu_{\frac{n}{2}}$ are the zeros of \eqref{eq:3.1a}. Moreover, if $\lambda_{\tau(2k-1)}$ are all distinct then the $k$th inequalities \eqref{eq:3.5a} are strict, and $\lambda_{\tau(2k-1)} = \lambda_{\tau(2k+1)}$ if and only if $\lambda_{\tau(2k-1)} = \mu_{k} = \lambda_{\tau(2k+1)}$.

\item[\textnormal{ii.}] $\lambda_{\sigma(2)} \leqslant \lambda_{\sigma(4)} \leqslant \ldots \leqslant \lambda_{\sigma(n+2)}$ are arranged in nondecreasing order by some bijection $\sigma$ then
\begin{equation}\label{eq:3.5b}
\lambda_{\sigma(2k)} \leqslant \nu_{k} \leqslant \lambda_{\sigma(2k+2)}
\end{equation}
for all $k =1,\ldots,\frac{n}{2}$, where $\nu_{1},\ldots,\nu_{\frac{n}{2}}$ are the zeros of \eqref{eq:3.1b}. Furthermore, if $\lambda_{\sigma(2k)}$ are all distinct then the $k$th inequalities \eqref{eq:3.5b} are strict, and $\lambda_{\sigma(2k)} = \lambda_{\sigma(2k+2)}$ if and only if $\lambda_{\sigma(2k)} = \nu_{k} = \lambda_{\sigma(2k+2)}$.
\end{itemize}
\end{subequations}

\medskip

\noindent \textnormal{(b)} If $n$ is odd and:

\begin{subequations}
\begin{itemize}
\item[\textnormal{i.}] $\lambda_{\tau(1)} \leqslant \lambda_{\tau(3)} \leqslant \ldots \leqslant \lambda_{\tau(n)} \leqslant \lambda_{\tau(n+2)}$ are arranged in nondecreasing order by some bijection $\tau$ then
\begin{equation}\label{eq:3.6a}
\lambda_{\tau(2k-1)} \leqslant \mu_{k} \leqslant \lambda_{\tau(2k+1)}
\end{equation}
for every $k=1,\ldots,\frac{n+1}{2}$, where $\mu_{1},\ldots,\mu_{\frac{n+1}{2}}$ are the zeros of \eqref{eq:3.2a}. Moreover, if $\lambda_{\tau(2k-1)}$ are all distinct then the $k$th inequalities \eqref{eq:3.5a} are strict, and $\lambda_{\tau(2k-1)} = \lambda_{\tau(2k+1)}$ if and only if $\lambda_{\tau(2k-1)} = \mu_{k} = \lambda_{\tau(2k+1)}$.

\item[\textnormal{ii.}] $\lambda_{\sigma(2)} \leqslant \lambda_{\sigma(4)} \leqslant \ldots \leqslant \lambda_{\sigma(n+1)}$ are arranged in nondecreasing order by some bijection $\sigma$ then
\begin{equation}\label{eq:3.6b}
\lambda_{\sigma(2k)} \leqslant \nu_{k} \leqslant \lambda_{\sigma(2k+2)}
\end{equation}
for each $k=1,\ldots,\frac{n-1}{2}$, where $\nu_{1},\ldots,\nu_{\frac{n-1}{2}}$ are the zeros of \eqref{eq:3.2b}. Furthermore, if $\lambda_{\sigma(2k)}$ are all distinct then the $k$th inequalities \eqref{eq:3.6b} are strict, and $\lambda_{\sigma(2k)} = \lambda_{\sigma(2k+2)}$ if and only if $\lambda_{\sigma(2k)} = \nu_{k} = \lambda_{\sigma(2k+2)}$.
\end{itemize}
\end{subequations}
\end{cor}

\begin{proof}
Since all assertions can be proven in the same way, we only prove i. of (a). Let $n \in \mathbb{N}$ be even and consider $\tau(2\xi - 1) = 1$ for some $1 \leqslant \xi \leqslant \frac{n}{2} + 1$. According to Theorem~\ref{thr1}, we have
\begin{align*}
f(t) &= \prod_{\substack{j=1}}^{\frac{n}{2}} (t - \lambda_{2j+1}) + (t - \lambda_{1}) \sum_{j=1}^{\frac{n}{2}} \frac{\sin^{2} \left[\frac{(2j + 1)\pi}{n + 3} \right]}{\sin^{2} \left(\frac{\pi}{n + 3} \right)} \prod_{\substack{\substack{m=1 \\ m \neq j}}}^{\frac{n}{2}} (t - \lambda_{2m+1}) \\
&= \prod_{\substack{j=1 \\ j \neq \xi}}^{\frac{n}{2} + 1} \left[t - \lambda_{\tau(2j-1)} \right] + \left[t - \lambda_{\tau(2\xi - 1)} \right] \sum_{\substack{j=1 \\ j \neq \xi}}^{\frac{n}{2} + 1} \frac{\sin^{2} \left[\frac{\tau(2j - 1)\pi}{n + 3} \right]}{\sin^{2} \left(\frac{\pi}{n + 3} \right)} \prod_{\substack{\substack{m=1 \\ m \neq \xi, j}}}^{\frac{n}{2} + 1} \left[t - \lambda_{\tau(2m-1)} \right]
\end{align*}
so that $f\left[\lambda_{\tau(2k-1)} \right] f\left[\lambda_{\tau(2k+1)} \right] \leqslant 0$ for all $k=1,\ldots,\frac{n}{2}$. Hence, $\lambda_{\tau(2k-1)} \leqslant \mu_{k} \leqslant \lambda_{\tau(2k+1)}$. If $\lambda_{\tau(2k-1)}$ are all distinct then $\lambda_{\tau(2k-1)} < \mu_{k} \leqslant \lambda_{\tau(2k+1)}$ or $\lambda_{\tau(2k-1)} \leqslant \mu_{k} < \lambda_{\tau(2k+1)}$. In the first case, if $\mu_{k} = \lambda_{\tau(2k+1)}$ then $f \left[\lambda_{\tau(2k+1)} \right] = 0$, implying $\lambda_{\tau(2k+1)} = \lambda_{\tau(2\zeta - 1)}$ for some $\zeta \neq k + 1$ which contradicts the hypothesis; the case $\lambda_{\tau(2k-1)} \leqslant \mu_{k} < \lambda_{\tau(2k+1)}$ is analogous. Hence, the $k$th inequalities \eqref{eq:3.5a} are strict.
\end{proof}

\begin{rem}
Let us observe that, in the assumptions of Corollary~\ref{cor1}, if $\lambda_{\tau(1)} > 0$ and $\lambda_{\sigma(2)} > 0$ then $\mathbf{H}_{n}$ is positive definite.
\end{rem}

\subsection{Eigenvectors}

\indent

From Lemma~\ref{lem3} we can get also an explicit representation for the eigenvectors of $\mathbf{H}_{n}$. 

\begin{thr}\label{thr2}
Let $n \in \mathbb{N}$, $a,b,c,d \in \mathbb{R}$, $\lambda_{k}$, $k=1,2,\ldots,n+2$ be given by \eqref{eq:2.3} and $\mathbf{H}_{n}$ the $n \times n$ matrix \eqref{eq:1.1}.

\medskip

\begin{subequations}
\noindent \textnormal{(a)} If $n$ is even, $\mathbf{u}, \mathbf{v}$ are given in \eqref{eq:2.4b}, $\mathbf{R}_{n}$ is the $n \times n$ matrix \eqref{eq:2.4c}, $\mathbf{P}_{n}$ is the $n \times n$ permutation matrix \eqref{eq:2.4d},
\begin{itemize}
\item[\textnormal{i.}] $\mu_{1}, \ldots, \mu_{\frac{n}{2}}$ are the zeros of \eqref{eq:3.1a} and $\lambda_{1}, \lambda_{3}, \ldots, \lambda_{n+1}$ are all distinct then
\begin{equation}\label{eq:3.7a}
\mathbf{R}_{n} \mathbf{P}_{n}^{\top} \left[
    \begin{array}{cc}
      \mathbf{I}_{\frac{n}{2}} + \mathbf{u} \mathbf{u}^{\top} & \mathbf{O} \\
      \mathbf{O} & \mathbf{I}_{\frac{n}{2}} + \mathbf{v} \mathbf{v}^{\top}
    \end{array}
    \right]
    \left[
    \begin{array}{c}
      \frac{\sin\left(\frac{3\pi}{n+3} \right)}{(\mu_{k} - \lambda_{3}) \sin \left(\frac{\pi}{n+3} \right)} \\[8pt]
      \frac{\sin\left(\frac{5\pi}{n+3} \right)}{(\mu_{k} - \lambda_{5}) \sin \left(\frac{\pi}{n+3} \right)} \\
      \vdots \\[2pt]
      \frac{\sin\left[\frac{(n+1)\pi}{n+3} \right]}{(\mu_{k} - \lambda_{n+1}) \sin \left(\frac{\pi}{n+3} \right)} \\[8pt]
      0 \\
      \vdots \\[2pt]
      0
    \end{array}
    \right]
\end{equation}
is an eigenvector of $\mathbf{H}_{n}$ associated to $\mu_{k}$, $k=1,\ldots, \frac{n}{2}$.

\item[\textnormal{ii.}] $\nu_{1}, \ldots, \nu_{\frac{n}{2}}$ are the zeros of \eqref{eq:3.1b} and $\lambda_{2}, \lambda_{4}, \ldots, \lambda_{n+2}$ are all distinct then
\begin{equation}\label{eq:3.7b}
\mathbf{R}_{n} \mathbf{P}_{n}^{\top} \left[
    \begin{array}{cc}
      \mathbf{I}_{\frac{n}{2}} + \mathbf{u} \mathbf{u}^{\top} & \mathbf{O} \\
      \mathbf{O} & \mathbf{I}_{\frac{n}{2}} + \mathbf{v} \mathbf{v}^{\top}
    \end{array}
    \right] \left[
    \begin{array}{c}
      0 \\
      \vdots \\[2pt]
      0 \\[4pt]
      \frac{\sin\left(\frac{2\pi}{n+3} \right)}{(\nu_{k} - \lambda_{2}) \sin \left(\frac{\pi}{n+3} \right)} \\[8pt]
      \frac{\sin\left(\frac{4\pi}{n+3} \right)}{(\nu_{k} - \lambda_{4}) \sin \left(\frac{\pi}{n+3} \right)} \\
      \vdots \\[2pt]
      \frac{\sin\left(\frac{n\pi}{n+3} \right)}{(\nu_{k} - \lambda_{n}) \sin \left(\frac{\pi}{n+3} \right)}
    \end{array}
    \right]
\end{equation}
is an eigenvector of $\mathbf{H}_{n}$ associated to $\nu_{k}$, $k=1,\ldots, \frac{n}{2}$.
\end{itemize}
\end{subequations}

\medskip

\begin{subequations}
\noindent \textnormal{(b)} If $n$ is odd, $\mathbf{u}, \mathbf{v}$ are given in \eqref{eq:2.5b}, $\mathbf{R}_{n}$ is the $n \times n$ matrix \eqref{eq:2.5c}, $\mathbf{P}_{n}$ is the $n \times n$ permutation matrix \eqref{eq:2.5d},
\begin{itemize}
\item[\textnormal{i.}] $\mu_{1}, \ldots, \mu_{\frac{n+1}{2}}$ are the zeros of \eqref{eq:3.1a} and $\lambda_{1}, \lambda_{3}, \ldots, \lambda_{n}, \lambda_{n+2}$ are all distinct then

\begin{equation}\label{eq:3.8a}
\mathbf{R}_{n} \mathbf{P}_{n} \left[
    \begin{array}{cc}
      \mathbf{I}_{\frac{n+1}{2}} + \mathbf{u} \mathbf{u}^{\top} & \mathbf{O} \\
      \mathbf{O} & \mathbf{I}_{\frac{n-1}{2}} + \mathbf{v} \mathbf{v}^{\top}
    \end{array}
    \right] \left[
    \begin{array}{c}
      \frac{\sin\left(\frac{3\pi}{n+3} \right)}{(\mu_{k} - \lambda_{3}) \sin \left(\frac{\pi}{n+3} \right)} \\[8pt]
      \frac{\sin\left(\frac{5\pi}{n+3} \right)}{(\mu_{k} - \lambda_{5}) \sin \left(\frac{\pi}{n+3} \right)} \\
      \vdots \\[2pt]
      \frac{\sin\left[\frac{(n+1)\pi}{n+3} \right]}{(\mu_{k} - \lambda_{n}) \sin \left(\frac{\pi}{n+3} \right)} \\[8pt]
      \frac{1}{\mu_{k} - \lambda_{n+2}} \\[5pt]
      0 \\
      \vdots \\[2pt]
      0
    \end{array}
    \right]
\end{equation}
is an eigenvector of $\mathbf{H}_{n}$ associated to $\mu_{k}$, $k=1,\ldots, \frac{n+1}{2}$.

\item[\textnormal{ii.}] $\nu_{1}, \ldots, \nu_{\frac{n-1}{2}}$ are the zeros of \eqref{eq:3.1b} and $\lambda_{2}, \lambda_{4}, \ldots, \lambda_{n+1}$ are all distinct then
\begin{equation}\label{eq:3.8b}
\mathbf{R}_{n} \mathbf{P}_{n} \left[
    \begin{array}{cc}
      \mathbf{I}_{\frac{n+1}{2}} + \mathbf{u} \mathbf{u}^{\top} & \mathbf{O} \\
      \mathbf{O} & \mathbf{I}_{\frac{n-1}{2}} + \mathbf{v} \mathbf{v}^{\top}
    \end{array}
    \right] \left[
    \begin{array}{c}
      0 \\
      \vdots \\[2pt]
      0 \\[3pt]
      \frac{\sin\left(\frac{2\pi}{n+3} \right)}{(\nu_{k} - \lambda_{2}) \sin \left[\frac{(n+1)\pi}{n+3} \right]} \\[8pt]
      \frac{\sin\left(\frac{4\pi}{n+3} \right)}{(\nu_{k} - \lambda_{4}) \sin \left[\frac{(n+1)\pi}{n+3} \right]} \\
      \vdots \\[2pt]
      \frac{\sin\left[\frac{(n-1)\pi}{n+3} \right]}{(\nu_{k} - \lambda_{n-1}) \sin \left[\frac{(n+1)\pi}{n+3} \right]}
    \end{array}
    \right]
\end{equation}
is an eigenvector of $\mathbf{H}_{n}$ associated to $\nu_{k}$, $k=1,\ldots, \frac{n-1}{2}$.
\end{itemize}
\end{subequations}
\end{thr}

\begin{proof}
Since both assertions can be proven in the same way, we only prove (a). Let $n \in \mathbb{N}$ be even.

\begin{itemize}
\item[i.] Supposing $\lambda_{1}, \lambda_{3}, \ldots, \lambda_{n+1}$ all distinct, Corollary~\ref{cor1} guarantees that the zeros of \eqref{eq:3.1a} are all simple. Setting
    \begin{gather*}
    \boldsymbol{\Upsilon}_{\frac{n}{2}} := \mathrm{diag}(\lambda_{3}, \lambda_{5}, \ldots, \lambda_{n+1}) \\
    \boldsymbol{\Delta}_{\frac{n}{2}} := \mathrm{diag}(\lambda_{2}, \lambda_{4}, \ldots, \lambda_{n})
    \end{gather*}
    we can rewrite the matricial equation $(\mu_{k} \mathbf{I}_{n} - \mathbf{H}_{n}) \mathbf{q} = \mathbf{0}$ as
    \begin{equation}\label{eq:3.9}
    \mathbf{R}_{n} \mathbf{P}_{n}^{\top} \left[
    \setlength{\extrarowheight}{2pt}
    \begin{array}{c|c}
      \mu_{k} \mathbf{I}_{\frac{n}{2}} - \boldsymbol{\Upsilon}_{\frac{n}{2}} + (\mu_{k} - \lambda_{1}) \mathbf{u} \mathbf{u}^{\top} & \mathbf{O} \\[2pt] \hline
      \mathbf{O} & \mu_{k} \mathbf{I}_{\frac{n}{2}} - \boldsymbol{\Delta}_{\frac{n}{2}} + (\mu_{k} - \lambda_{n+2}) \mathbf{v} \mathbf{v}^{\top}
    \end{array}
    \right] \mathbf{P}_{n} \mathbf{R}_{n} \mathbf{q} = \mathbf{0}
    \end{equation}
    with $\mathbf{u}, \mathbf{v}$ given by \eqref{eq:2.4b}, $\mathbf{R}_{n}$ defined in \eqref{eq:2.4c} and $\mathbf{P}_{n}$ the permutation matrix \eqref{eq:2.4d}. Thus,
    \begin{gather*}
    \left[\mu_{k} \mathbf{I}_{\frac{n}{2}} - \boldsymbol{\Upsilon}_{\frac{n}{2}} + (\mu_{k} - \lambda_{1}) \mathbf{u} \mathbf{u}^{\top} \right] \mathbf{q}_{1} = \mathbf{0}, \\
    \left[\mu_{k} \mathbf{I}_{\frac{n}{2}} - \boldsymbol{\Delta}_{\frac{n}{2}} + (\mu_{k} - \lambda_{n+2}) \mathbf{v} \mathbf{v}^{\top} \right] \mathbf{q}_{2} = \mathbf{0}, \\
    \left[
    \begin{array}{c}
      \mathbf{q}_{1} \\
      \mathbf{q}_{2}
    \end{array}
    \right] = \mathbf{P}_{n} \mathbf{R}_{n} \mathbf{q}
    \end{gather*}
    that is,
    \begin{gather*}
    \mathbf{q}_{1} = \alpha (\mu_{k} \mathbf{I}_{\frac{n}{2}} - \boldsymbol{\Upsilon}_{\frac{n}{2}})^{-1} \mathbf{u}, \\
    \mathbf{q}_{2} = \mathbf{0}
    \end{gather*}
    for $\alpha \neq 0$ and
    \begin{equation*}
    \mathbf{q} =  \mathbf{R}_{n}^{-1} \mathbf{P}_{n}^{\top}  \left[
    \begin{array}{c}
      \alpha (\mu_{k} \mathbf{I}_{\frac{n}{2}} - \boldsymbol{\Upsilon}_{\frac{n}{2}})^{-1} \mathbf{u} \\
      \mathbf{0}
    \end{array}
    \right]
    \end{equation*}
    is a nontrivial solution of \eqref{eq:3.9}, and so an eigenvector of $\mathbf{H}_{n}$ associated to the eigenvalue $\mu_{k}$. Since
    \begin{equation}\label{eq:3.10}
    \begin{split}
    \mathbf{R}_{n}^{-1} \mathbf{P}_{n}^{\top} &= \mathbf{R}_{n} \mathbf{R}_{n}^{-2} \mathbf{P}_{n}^{\top} \\
    &= \mathbf{R}_{n} \left(\mathbf{I}_{n} - \mathbf{x} \mathbf{x}^{\top} - \mathbf{y} \mathbf{y}^{\top} \right)^{-1} \mathbf{P}_{n}^{\top} \\
    &= \mathbf{R}_{n} \mathbf{P}_{n}^{\top} \left[\mathbf{P}_{n} \left(\mathbf{I}_{n} - \mathbf{x} \mathbf{x}^{\top} - \mathbf{y} \mathbf{y}^{\top} \right) \mathbf{P}_{n}^{\top} \right]^{-1} \\
    &= \mathbf{R}_{n} \mathbf{P}_{n}^{\top} \left[
    \begin{array}{cc}
      \mathbf{I}_{\frac{n}{2}} + \mathbf{u} \mathbf{u}^{\top} & \mathbf{O} \\
      \mathbf{O} & \mathbf{I}_{\frac{n}{2}} + \mathbf{v} \mathbf{v}^{\top}
    \end{array}
    \right],
    \end{split}
    \end{equation}
    where $\mathbf{x}, \mathbf{y}$ are given by \eqref{eq:2.6}, the result follows choosing $\alpha = 1$.

\medskip

\item[ii.] Considering $\lambda_{2}, \lambda_{4}, \ldots, \lambda_{n+2}$ all distinct, the zeros of \eqref{eq:3.1b} are all simple according to Corollary~\ref{cor1}. Hence,
    \begin{equation*}
    \mathbf{R}_{n} \mathbf{P}_{n}^{\top} \left[
    \setlength{\extrarowheight}{2pt}
    \begin{array}{c|c}
      \nu_{k} \mathbf{I}_{\frac{n}{2}} - \boldsymbol{\Upsilon}_{\frac{n}{2}} + (\nu_{k} - \lambda_{1}) \mathbf{u} \mathbf{u}^{\top} & \mathbf{O} \\[2pt] \hline
      \mathbf{O} & \nu_{k} \mathbf{I}_{\frac{n}{2}} - \boldsymbol{\Delta}_{\frac{n}{2}} + (\nu_{k} - \lambda_{n+2}) \mathbf{v} \mathbf{v}^{\top}
    \end{array}
    \right] \mathbf{P}_{n} \mathbf{R}_{n} \mathbf{q} = \mathbf{0}
    \end{equation*}
    leads to
    \begin{equation*}
    \mathbf{q} =  \mathbf{R}_{n}^{-1} \mathbf{P}_{n}^{\top}  \left[
    \begin{array}{c}
      \mathbf{0} \\
      \alpha (\nu_{k} \mathbf{I}_{\frac{n}{2}} - \boldsymbol{\Delta}_{\frac{n}{2}})^{-1} \mathbf{v}
    \end{array}
    \right],
    \end{equation*}
    for $\alpha \neq 0$, which is an eigenvector of $\mathbf{H}_{n}$ associated to the eigenvalue $\nu_{k}$. The conclusion follows from \eqref{eq:3.10}.
\end{itemize}
\end{proof}

\begin{rem}
Recall that if $n$ is even and some $\lambda_{2k-1}$ is a root of \eqref{eq:3.1a} (and \emph{a fortiori} an eigenvalue of $\mathbf{H}_{n}$) then $\lambda_{2k-1} = \lambda_{2\ell-1}$ for a suitable $\ell \neq k$ according to Corollary~\ref{cor1}, so that
\begin{equation*}
\mathbf{R}_{n} \mathbf{P}_{n}^{\top} \left[
    \begin{array}{cc}
      \mathbf{I}_{\frac{n}{2}} + \mathbf{u} \mathbf{u}^{\top} & \mathbf{O} \\
      \mathbf{O} & \mathbf{I}_{\frac{n}{2}} + \mathbf{v} \mathbf{v}^{\top}
    \end{array}
    \right] \left[
\begin{array}{c}
\mathbf{q}_{1} \\
\mathbf{0}
\end{array}
\right],
\end{equation*}
where $\mathbf{u}, \mathbf{v}$ are given in \eqref{eq:2.4b}, $\mathbf{R}_{n}$ is the $n \times n$ matrix \eqref{eq:2.4c}, $\mathbf{P}_{n}$ is the $n \times n$ permutation matrix \eqref{eq:2.4d},
\begin{equation*}
\mathbf{q}_{1} = \left\{
\begin{array}{l}
  \sin \left[\frac{(2\ell - 1) \pi}{n + 3} \right] \mathbf{i}_{k} - \sin \left[\frac{(2k - 1) \pi}{n + 3} \right] \mathbf{i}_{\ell} \; \; \textnormal{if} \; \; k,\ell \neq 1 \\[10pt]
  \mathbf{i}_{\ell} \; \; \textnormal{if} \; \; k = 1 \\[3pt]
  \mathbf{i}_{k} \; \; \textnormal{if} \; \; \ell = 1
\end{array}
\right.,
\end{equation*}
and $\mathbf{i}_{k}$ denotes the $k$th column of $\mathbf{I}_{\frac{n}{2}}$, is an eigenvector of $\mathbf{H}_{n}$ associated to $\lambda_{2k-1}$; if a $\lambda_{2k}$ is root of \eqref{eq:3.1b} then $\lambda_{2k} = \lambda_{2\ell}$ for some $\ell \neq k$ and
\begin{equation*}
\mathbf{R}_{n} \mathbf{P}_{n}^{\top} \left[
    \begin{array}{cc}
      \mathbf{I}_{\frac{n}{2}} + \mathbf{u} \mathbf{u}^{\top} & \mathbf{O} \\
      \mathbf{O} & \mathbf{I}_{\frac{n}{2}} + \mathbf{v} \mathbf{v}^{\top}
    \end{array}
    \right] \left[
\begin{array}{c}
\mathbf{0} \\
\mathbf{q}_{2}
\end{array}
\right],
\end{equation*}
with $\mathbf{u}, \mathbf{v}$ given by \eqref{eq:2.4b}, $\mathbf{R}_{n}$ the $n \times n$ matrix \eqref{eq:2.4c}, $\mathbf{P}_{n}$ the $n \times n$ permutation matrix \eqref{eq:2.4d} and
\begin{equation*}
\mathbf{q}_{2} = \left\{
\begin{array}{l}
  \sin \left(\frac{2\ell \pi}{n + 3} \right) \mathbf{i}_{k} - \sin \left(\frac{2k \pi}{n + 3} \right) \mathbf{i}_{\ell} \; \; \textnormal{if} \; \; k,\ell \neq \frac{n+2}{2} \\[10pt]
  \mathbf{i}_{\ell} \; \; \textnormal{if} \; \; k = \frac{n+2}{2} \\[7pt]
  \mathbf{i}_{k} \; \; \textnormal{if} \; \; \ell = \frac{n+2}{2}
\end{array}
\right.
\end{equation*}
is an eigenvector of $\mathbf{H}_{n}$ associated to $\lambda_{2k}$. Analogously, if $n$ is odd and some $\lambda_{2k-1}$ is a root of \eqref{eq:3.2a} then $\lambda_{2k-1} = \lambda_{2\ell-1}$ for a suitable $\ell \neq k$ and
\begin{equation*}
\mathbf{R}_{n} \mathbf{P}_{n} \left[
    \begin{array}{cc}
      \mathbf{I}_{\frac{n+1}{2}} + \mathbf{u} \mathbf{u}^{\top} & \mathbf{O} \\
      \mathbf{O} & \mathbf{I}_{\frac{n-1}{2}} + \mathbf{v} \mathbf{v}^{\top}
    \end{array}
    \right] \left[
\begin{array}{c}
\mathbf{q}_{1} \\
\mathbf{0}
\end{array}
\right],
\end{equation*}
where $\mathbf{u}, \mathbf{v}$ are given in \eqref{eq:2.5b}, $\mathbf{R}_{n}$ is the $n \times n$ matrix \eqref{eq:2.5c}, $\mathbf{P}_{n}$ is the $n \times n$ permutation matrix \eqref{eq:2.5d},
\begin{equation*}
\mathbf{q}_{1} = \left\{
\begin{array}{l}
  \sin \left[\frac{(2\ell - 1) \pi}{n + 3} \right] \mathbf{i}_{k} - \sin \left[\frac{(2k - 1)\pi}{n + 3} \right] \mathbf{i}_{\ell} \; \; \textnormal{if} \; \; k,\ell \neq 1 \\[10pt]
  \mathbf{i}_{\ell} \; \; \textnormal{if} \; \; k = 1 \\[3pt]
  \mathbf{i}_{k} \; \; \textnormal{if} \; \; \ell = 1
\end{array}
\right.,
\end{equation*}
and $\mathbf{i}_{k}$ denotes the $k$th column of $\mathbf{I}_{\frac{n+1}{2}}$, is an eigenvector of $\mathbf{H}_{n}$ associated to $\lambda_{2k-1}$; if a $\lambda_{2k}$ is a root of \eqref{eq:3.2b} then $\lambda_{2k} = \lambda_{2\ell}$ for some $\ell \neq k$ and
\begin{equation*}
\mathbf{R}_{n} \mathbf{P}_{n} \left[
    \begin{array}{cc}
      \mathbf{I}_{\frac{n+1}{2}} + \mathbf{u} \mathbf{u}^{\top} & \mathbf{O} \\
      \mathbf{O} & \mathbf{I}_{\frac{n-1}{2}} + \mathbf{v} \mathbf{v}^{\top}
    \end{array}
    \right] \left[
\begin{array}{c}
\mathbf{0} \\
\mathbf{q}_{2}
\end{array}
\right],
\end{equation*}
with $\mathbf{u}, \mathbf{v}$ given by \eqref{eq:2.5b}, $\mathbf{R}_{n}$ the $n \times n$ matrix \eqref{eq:2.5c}, $\mathbf{P}_{n}$ the $n \times n$ permutation matrix \eqref{eq:2.5d},
\begin{equation*}
\mathbf{q}_{2} = \left\{
\begin{array}{l}
  \sin \left(\frac{2\ell \pi}{n + 3} \right) \mathbf{i}_{k} - \sin \left(\frac{2k \pi}{n + 3} \right) \mathbf{i}_{\ell}, \; \; \textnormal{if} \; \; k,\ell \neq \frac{n+1}{2} \\[10pt]
  \mathbf{i}_{\ell}, \; \; \textnormal{if} \; \; k = \frac{n+1}{2} \\[3pt]
  \mathbf{i}_{k}, \; \; \textnormal{if} \; \; \ell = \frac{n+1}{2}
\end{array}
\right.
\end{equation*}
and $\mathbf{i}_{k}$ denotes the $k$th column of $\mathbf{I}_{\frac{n-1}{2}}$, is an eigenvector of $\mathbf{H}_{n}$ associated to $\lambda_{2k}$.
\end{rem}

\section{Integer powers of $\mathbf{H}_{n}$}

\indent

The decomposition in Lemma~\ref{lem3} can still be used to compute explicitly the inverse of $\mathbf{H}_{n}$.

\begin{thr}\label{thr3}
Let $n \in \mathbb{N}$, $a,b,c,d \in \mathbb{C}$, $\lambda_{k}$, $k = 1,2,\ldots,n+2$ be given by \eqref{eq:2.3} and $\mathbf{H}_{n}$ the $n \times n$ matrix \eqref{eq:1.1}. If $\lambda_{k} \neq 0$ for every $k = 1,2,\ldots,n+2$, $\mathbf{H}_{n}$ is nonsingular and:

\medskip

\noindent \textnormal{(a)} $n$ is even then
\begin{subequations}
\begin{equation*}
\mathbf{H}_{n}^{-1} = \mathbf{R}_{n} \mathbf{P}_{n}^{\top} \left[
\begin{array}{cc}
\mathbf{U}_{\frac{n}{2}} & \mathbf{O} \\
\mathbf{O} & \mathbf{V}_{\frac{n}{2}}
\end{array}
\right] \mathbf{P}_{n} \mathbf{R}_{n}
\end{equation*}
where $\mathbf{R}_{n}$ is the $n \times n$ matrix \eqref{eq:2.4c}, $\mathbf{P}_{n}$ is the $n \times n$ permutation matrix \eqref{eq:2.4d},
\begin{equation}\label{eq:4.1a}
\begin{split}
\mathbf{U}_{\frac{n}{2}} = \left(\mathbf{I}_{\frac{n}{2}} + \mathbf{u} \mathbf{u}^{\top} \right) & \bigg[\mathrm{diag} \left(\lambda_{3}^{-1},\lambda_{5}^{-1},\ldots,\lambda_{n+1}^{-1} \right) - \tfrac{\lambda_{1}}{1 + \lambda_{1} \mathbf{u}^{\top} \mathrm{diag} \left(\lambda_{3}^{-1},\lambda_{5}^{-1},\ldots,\lambda_{n+1}^{-1} \right) \mathbf{u}} \cdot \bigg. \\
& \bigg. \mathrm{diag} \left(\lambda_{3}^{-1},\lambda_{5}^{-1},\ldots,\lambda_{n+1}^{-1} \right) \mathbf{u} \mathbf{u}^{\top} \mathrm{diag} \left(\lambda_{3}^{-1},\lambda_{5}^{-1},\ldots,\lambda_{n+1}^{-1} \right) \bigg] \left(\mathbf{I}_{\frac{n}{2}} + \mathbf{u} \mathbf{u}^{\top} \right)
\end{split}
\end{equation}
and
\begin{equation}\label{eq:4.1b}
\begin{split}
\mathbf{V}_{\frac{n}{2}} = \left(\mathbf{I}_{\frac{n}{2}} + \mathbf{v} \mathbf{v}^{\top} \right) & \bigg[\mathrm{diag} \left(\lambda_{2}^{-1},\lambda_{4}^{-1},\ldots,\lambda_{n}^{-1} \right) - \tfrac{\lambda_{n+2}}{1 + \lambda_{n+2} \mathbf{v}^{\top} \mathrm{diag} \left(\lambda_{2}^{-1},\lambda_{4}^{-1},\ldots,\lambda_{n}^{-1} \right) \mathbf{v}} \cdot \bigg. \\
& \bigg. \mathrm{diag} \left(\lambda_{2}^{-1},\lambda_{4}^{-1},\ldots,\lambda_{n}^{-1} \right) \mathbf{v} \mathbf{v}^{\top} \mathrm{diag} \left(\lambda_{2}^{-1},\lambda_{4}^{-1},\ldots,\lambda_{n}^{-1} \right) \bigg] \left(\mathbf{I}_{\frac{n}{2}} + \mathbf{v} \mathbf{v}^{\top} \right)
\end{split}
\end{equation}
with $\mathbf{u}, \mathbf{v}$ given by \eqref{eq:2.4b}.
\end{subequations}

\medskip

\noindent \textnormal{(b)} $n$ is odd then
\begin{subequations}
\begin{equation*}
\mathbf{H}_{n}^{-1} = \mathbf{R}_{n} \mathbf{P}_{n} \left[
\begin{array}{cc}
\mathbf{U}_{\frac{n+1}{2}} & \mathbf{O} \\
\mathbf{O} & \mathbf{V}_{\frac{n-1}{2}}
\end{array}
\right] \mathbf{P}_{n}^{\top} \mathbf{R}_{n}^{\top}
\end{equation*}
where $\mathbf{R}_{n}$ is the $n \times n$ matrix \eqref{eq:2.5c}, $\mathbf{P}_{n}$ is the $n \times n$ permutation matrix \eqref{eq:2.5d},
\begin{equation}\label{eq:4.2a}
\begin{split}
\mathbf{U}_{\frac{n+1}{2}} & = \left(\mathbf{I}_{\frac{n+1}{2}} + \mathbf{u} \mathbf{u}^{\top} \right) \bigg[\mathrm{diag} \left(\lambda_{3}^{-1},\lambda_{5}^{-1},\ldots,\lambda_{n}^{-1},\lambda_{n+2}^{-1} \right) - \tfrac{\lambda_{1}}{1 + \lambda_{1} \mathbf{u}^{\top} \mathrm{diag} \left(\lambda_{3}^{-1},\lambda_{5}^{-1},\ldots,\lambda_{n}^{-1}, \lambda_{n+2}^{-1} \right) \mathbf{u}} \cdot \bigg. \\
& \hspace*{0.8cm} \bigg. \mathrm{diag} \left(\lambda_{3}^{-1},\lambda_{5}^{-1},\ldots,\lambda_{n}^{-1}, \lambda_{n+2}^{-1} \right) \mathbf{u} \mathbf{u}^{\top} \mathrm{diag} \left(\lambda_{3}^{-1},\lambda_{5}^{-1},\ldots,\lambda_{n}^{-1},\lambda_{n+2}^{-1} \right) \bigg] \left(\mathbf{I}_{\frac{n+1}{2}} + \mathbf{u} \mathbf{u}^{\top} \right)
\end{split}
\end{equation}
and
\begin{equation}\label{eq:4.2b}
\begin{split}
\mathbf{V}_{\frac{n-1}{2}} & = \left(\mathbf{I}_{\frac{n-1}{2}} + \mathbf{v} \mathbf{v}^{\top} \right) \bigg[\mathrm{diag} \left(\lambda_{2}^{-1},\lambda_{4}^{-1},\ldots,\lambda_{n-1}^{-1} \right) - \tfrac{\lambda_{n+1}}{1 + \lambda_{n+1} \mathbf{v}^{\top} \mathrm{diag} \left(\lambda_{2}^{-1},\lambda_{4}^{-1},\ldots,\lambda_{n-1}^{-1} \right) \mathbf{v}} \cdot \bigg. \\
& \hspace*{2.0cm} \bigg. \mathrm{diag} \left(\lambda_{2}^{-1},\lambda_{4}^{-1},\ldots,\lambda_{n-1}^{-1} \right) \mathbf{v} \mathbf{v}^{\top} \mathrm{diag} \left(\lambda_{2}^{-1},\lambda_{4}^{-1},\ldots,\lambda_{n-1}^{-1} \right) \bigg] \left(\mathbf{I}_{\frac{n-1}{2}} + \mathbf{v} \mathbf{v}^{\top} \right)
\end{split}
\end{equation}
with $\mathbf{u}, \mathbf{v}$ given by \eqref{eq:2.5b}.
\end{subequations}
\end{thr}

\begin{proof}
Consider $n \in \mathbb{N}$ even, $a,b,c,d \in \mathbb{C}$ $\lambda_{k} \neq 0$, $k = 1,2,\ldots,n+2$ and $\mathbf{H}_{n}$ nonsingular. According to Sherman-Morrison-Woodbury formula (see \cite{Horn13}, page $19$ or \cite{Harville97}, page $424$), we have
\begin{equation*}
\left(\boldsymbol{\Upsilon}_{\frac{n}{2}} + \lambda_{1} \mathbf{u} \mathbf{u}^{\top} \right)^{-1} = \boldsymbol{\Upsilon}_{\frac{n}{2}}^{-1} - \lambda_{1} \left(1 + \lambda_{1} \mathbf{u}^{\top} \boldsymbol{\Upsilon}_{\frac{n}{2}}^{-1} \mathbf{u} \right)^{-1} \boldsymbol{\Upsilon}_{\frac{n}{2}}^{-1} \mathbf{u} \mathbf{u}^{\top} \boldsymbol{\Lambda}_{\frac{n}{2}}^{-1},
\end{equation*}
where $\boldsymbol{\Upsilon}_{\frac{n}{2}} := \mathrm{diag} \left(\lambda_{3},\lambda_{5},\ldots,\lambda_{n+1} \right)$, $\mathbf{u}$ is given in \eqref{eq:2.4a} and $\mathbf{U}_{\frac{n}{2}}$ is the matrix \eqref{eq:4.1a}. Similarly, we obtain
\begin{equation*}
\left(\boldsymbol{\Delta}_{\frac{n}{2}} + \lambda_{n+2} \mathbf{v} \mathbf{v}^{\top} \right)^{-1} = \boldsymbol{\Delta}_{\frac{n}{2}}^{-1} - \lambda_{n+2} \left(1 + \lambda_{n+2} \mathbf{v}^{\top} \boldsymbol{\Delta}_{\frac{n}{2}}^{-1} \mathbf{v} \right)^{-1} \boldsymbol{\Delta}_{\frac{n}{2}}^{-1} \mathbf{v} \mathbf{v}^{\top} \boldsymbol{\Delta}_{\frac{n}{2}}^{-1},
\end{equation*}
where $\boldsymbol{\Delta}_{\frac{n}{2}} := \mathrm{diag} \left(\lambda_{2},\lambda_{4},\ldots,\lambda_{n} \right)$, $\mathbf{v}$ is given in \eqref{eq:2.4a} and $\mathbf{V}_{\frac{n}{2}}$ is the matrix \eqref{eq:4.1b}. Hence, the decomposition of Lemma~\ref{lem3}, 8.5b of \cite{Harville97} (page $88$), identity \eqref{eq:3.10} and
\begin{equation*}
\begin{split}
\mathbf{P}_{n} \mathbf{R}_{n}^{-1} &= \mathbf{P}_{n} \mathbf{R}_{n}^{-2} \mathbf{R}_{n} \\
&=\mathbf{P}_{n} \left(\mathbf{I}_{n} - \mathbf{x} \mathbf{x}^{\top} - \mathbf{y} \mathbf{y}^{\top} \right)^{-1}  \mathbf{R}_{n} \\
&= \left[
\begin{array}{cc}
    \mathbf{I}_{\frac{n}{2}} + \mathbf{u} \mathbf{u}^{\top} & \mathbf{O} \\
    \mathbf{O} & \mathbf{I}_{\frac{n}{2}} + \mathbf{v} \mathbf{v}^{\top}
\end{array}
\right] \mathbf{P}_{n} \mathbf{R}_{n},
\end{split}
\end{equation*}
where $\mathbf{x}, \mathbf{y}$ are given by \eqref{eq:2.6}, establish the thesis in (a). The proof of (b) is analogous and so will be omitted.
\end{proof}

\begin{rem}
Note that if $n$ is even, $\lambda_{k} \neq 0$ for all $k = 2,3,\ldots,n+1$,
\begin{equation*}
\frac{\lambda_{1}}{\sin^{2} \left(\frac{\pi}{n+3} \right)} \sum_{j=1}^{\frac{n}{2}} \frac{\sin^{2} \left[\frac{(2j+1)\pi}{n+3} \right]}{\lambda_{2j+1}} \neq -1 \; \; \textnormal{and} \; \; \frac{\lambda_{n+2}}{\sin^{2} \left(\frac{\pi}{n+3} \right)} \sum_{j=1}^{\frac{n}{2}} \frac{\sin^{2} \left(\frac{2j\pi}{n+3} \right)}{\lambda_{2j}} \neq -1
\end{equation*}
then $\mathbf{H}_{n}$ is nonsingular. Similarly, if $n$ is odd, $\lambda_{k} \neq 0$ for all $k = 2,3,\ldots,n,n+2$,
\begin{equation*}
\frac{\lambda_{1}}{\sin^{2} \left(\frac{\pi}{n+3} \right)} \sum_{j=1}^{\frac{n+1}{2}} \frac{\sin^{2} \left[\frac{(2j+1)\pi}{n+3} \right]}{\lambda_{2j+1}} \neq -1 \; \; \textnormal{and} \; \; \frac{\lambda_{n+1}}{\sin^{2} \left[\frac{(n+1)\pi}{n+3} \right]} \sum_{j=1}^{\frac{n-1}{2}} \frac{\sin^{2} \left(\frac{2j\pi}{n+3} \right)}{\lambda_{2j}} \neq -1
\end{equation*}
then $\mathbf{H}_{n}$ is a nonsingular matrix.
\end{rem}

Let us point out that the previous Theorems~\ref{thr1} and~\ref{thr2} lead to an immediate eigenvalue decomposition with orthogonal eigenvector matrix for any real anti-heptadiagonal persymmetric Hankel matrix having all the eigenvalues distinct. Indeed, admitting $\mathbf{H}_{n}$ in \eqref{eq:1.1} real with its eigenvalues all different it is well-known that the set of its eigenvectors is orthogonal (see, for instance, \cite{Ford15} page $109$); hence, normalizing each of them we can obtain an orthonormal set, and so an orthogonal eigenvector matrix (see \cite{Williams14}, page $257$). However, it is not guaranteed, in general, that the eigenvalues of $\mathbf{H}_{n}$ are all different even assuming the distinctness of every $\lambda_{2k-1}$ and all $\lambda_{2k}$, as the following examples show:
\begin{equation*}
\mathbf{H}_{4} = \left[
\begin{array}{cccc}
  1 & 0 & 1 & 0 \\
  0 & 1 & 0 & 1 \\
  1 & 0 & 1 & 0 \\
  0 & 1 & 0 & 1
\end{array}
\right] \quad \textnormal{or} \quad \mathbf{H}_{5} = \left[
\begin{array}{ccccc}
  0 & 1 & 0 & 0 & 0 \\
  1 & 0 & 0 & 0 & 0 \\
  0 & 0 & 0 & 0 & 0 \\
  0 & 0 & 0 & 0 & 1 \\
  0 & 0 & 0 & 1 & 0
\end{array}
\right].
\end{equation*}
It is possible to establish an alternative diagonalization for $\mathbf{H}_{n}$ avoiding its eigenvalues. The following auxiliary result is the central key in its accomplishment.

\begin{lem}\label{lem4}
Let $n \in \mathbb{N}$, $a,b,c,d \in \mathbb{C}$, $\lambda_{k}$, $k = 1,2,\ldots,n+2$ be given by \eqref{eq:2.3} and $\mathbf{H}_{n}$ the $n \times n$ matrix \eqref{eq:1.1}.

\medskip

\noindent \textnormal{(a)} If $n$ is even and:

\begin{subequations}
\begin{itemize}
\item[\textnormal{i.}] $\mathbf{u}$ is given in \eqref{eq:2.4b} then the eigenvalues of
\begin{equation}\label{eq:4.3a}
\mathrm{diag} \left(\lambda_{3},\lambda_{5},\ldots,\lambda_{n+1} \right) + \frac{4 \sin^{2} \left(\frac{\pi}{n+3} \right)}{n + 3} \mathrm{diag} \left(\lambda_{1} - \lambda_{3},\lambda_{1} - \lambda_{5},\ldots,\lambda_{1} - \lambda_{n+1} \right) \mathbf{u} \mathbf{u}^{\top}
\end{equation}
are the zeros of the polynomial
\begin{equation}\label{eq:4.3b}
F(t) = \prod_{j=1}^{\frac{n}{2}}(t - \lambda_{2j+1}) + \frac{4}{n + 3} \sum_{j=1}^{\frac{n}{2}} \sin^{2} \left[\frac{(2j + 1) \pi}{n + 3} \right] (\lambda_{2j+1} - \lambda_{1}) \prod_{\substack{m=1 \\ m \neq j}}^{\frac{n}{2}}(t - \lambda_{2m+1}).
\end{equation}
Moreover, if $\lambda_{1},\lambda_{3},\ldots,\lambda_{n+1}$ are all real and distinct then the eigenvalues $\phi_{1}, \phi_{2},\ldots, \phi_{\frac{n}{2}}$ of \eqref{eq:4.3b} are all simple and
\begin{equation}\label{eq:4.3c}
\mathbf{f}_{k} := \left[
\begin{array}{c}
\frac{(\lambda_{3} - \lambda_{1}) \sin\left(\frac{3\pi}{n + 3} \right)}{\lambda_{3} - \phi_{k}} \\
\frac{(\lambda_{5} - \lambda_{1}) \sin\left(\frac{5\pi}{n + 3} \right)}{\lambda_{5} - \phi_{k}} \\
\vdots \\
\frac{(\lambda_{n+1} - \lambda_{1}) \sin\left[\frac{(n + 1)\pi}{n + 3} \right]}{\lambda_{n+1} - \phi_{k}}
\end{array}
\right]
\end{equation}
is an eigenvector associated to $\phi_{k}$, $k = 1,\ldots,\tfrac{n}{2}$.

\item[\textnormal{ii.}] $\mathbf{v}$ is given in \eqref{eq:2.4b} then the eigenvalues of
\begin{equation}\label{eq:4.3d}
\mathrm{diag} \left(\lambda_{2},\lambda_{4},\ldots,\lambda_{n} \right) + \frac{4 \sin^{2} \left(\frac{\pi}{n+3} \right)}{n + 3} \mathrm{diag} \left(\lambda_{n+2} - \lambda_{2},\lambda_{n+2} - \lambda_{4},\ldots,\lambda_{n+2} - \lambda_{n} \right) \mathbf{v} \mathbf{v}^{\top}
\end{equation}
are the zeros of the polynomial
\begin{equation}\label{eq:4.3e}
G(t) = \prod_{j=1}^{\frac{n}{2}}(t - \lambda_{2j}) + \frac{4}{n + 3} \sum_{j=1}^{\frac{n}{2}} \sin^{2} \left(\frac{2j \pi}{n + 3} \right) (\lambda_{2j} - \lambda_{n+2}) \prod_{\substack{m=1 \\ m \neq j}}^{\frac{n}{2}}(t - \lambda_{2m}).
\end{equation}
Furthermore, if $\lambda_{2},\lambda_{4},\ldots,\lambda_{n+2}$ are all real, nonzero and different then the eigenvalues $\psi_{1}, \psi_{2},\ldots, \psi_{\frac{n}{2}}$ of \eqref{eq:4.3e} are all simple and
\begin{equation}\label{eq:4.3f}
\mathbf{g}_{k} := \left[
\begin{array}{c}
\frac{(\lambda_{2} - \lambda_{n+2}) \sin\left(\frac{2\pi}{n + 3} \right)}{\lambda_{2} - \psi_{k}} \\
\frac{(\lambda_{4} - \lambda_{n+2})  \sin\left(\frac{4\pi}{n + 3} \right)}{\lambda_{4} - \psi_{k}} \\
\vdots \\
\frac{(\lambda_{n} - \lambda_{n+2}) \sin\left(\frac{n\pi}{n + 3} \right)}{\lambda_{n} - \psi_{k}}
\end{array}
\right]
\end{equation}
is an eigenvector associated to $\psi_{k}$, $k = 1,\ldots,\tfrac{n}{2}$.
\end{itemize}
\end{subequations}

\medskip

\noindent \textnormal{(b)} If $n$ is odd and:

\begin{subequations}
\begin{itemize}
\item[\textnormal{i.}] $\mathbf{u}$ is given in \eqref{eq:2.5b} then the eigenvalues of
\begin{equation}\label{eq:4.4a}
\mathrm{diag} \left(\lambda_{3},\lambda_{5},\ldots,\lambda_{n}, \lambda_{n+2} \right) + \frac{4 \sin^{2} \left(\frac{\pi}{n+3} \right)}{n + 3} \mathrm{diag} \left(\lambda_{1} - \lambda_{3},\lambda_{1} - \lambda_{5},\ldots,\lambda_{1} - \lambda_{n}, \lambda_{1} - \lambda_{n+2} \right) \mathbf{u} \mathbf{u}^{\top}
\end{equation}
are the zeros of the polynomial
\begin{equation}\label{eq:4.4b}
F(t) = \prod_{j=1}^{\frac{n+1}{2}}(t - \lambda_{2j+1}) + \frac{4}{n + 3} \sum_{j=1}^{\frac{n+1}{2}} \sin^{2} \left[\frac{(2j + 1) \pi}{n + 3} \right] (\lambda_{2j+1} - \lambda_{1}) \prod_{\substack{m=1 \\ m \neq j}}^{\frac{n+1}{2}}(t - \lambda_{2m+1})
\end{equation}
Moreover, if $\lambda_{1},\lambda_{3},\ldots,\lambda_{n},\lambda_{n+2}$ are all real and distinct then the eigenvalues $\phi_{1}, \phi_{2},\ldots, \phi_{\frac{n+1}{2}}$ of \eqref{eq:4.4b} are all simple and
\begin{equation}\label{eq:4.4c}
\mathbf{f}_{k} := \left[
\begin{array}{c}
\frac{(\lambda_{3} - \lambda_{1}) \sin\left(\frac{3\pi}{n + 3} \right)}{\lambda_{3} - \phi_{k}} \\
\frac{(\lambda_{5} - \lambda_{1}) \sin\left(\frac{5\pi}{n + 3} \right)}{\lambda_{5} - \phi_{k}} \\
\vdots \\
\frac{(\lambda_{n} - \lambda_{1}) \sin\left(\frac{n\pi}{n + 3} \right)}{\lambda_{n} - \phi_{k}} \\
\frac{(\lambda_{n+2} - \lambda_{1}) \sin\left[\frac{(n + 2)\pi}{n + 3} \right]}{\lambda_{n+2} - \phi_{k}}
\end{array}
\right]
\end{equation}
is an eigenvector associated to $\phi_{k}$, $k = 1,\ldots,\tfrac{n}{2}$.

\item[\textnormal{ii.}] $\mathbf{v}$ is given in \eqref{eq:2.5b} then the eigenvalues of
\begin{equation}\label{eq:4.4d}
\mathrm{diag} \left(\lambda_{2},\lambda_{4},\ldots,\lambda_{n-1} \right) + \frac{4 \sin^{2} \left[\frac{(n + 1)\pi}{n+3} \right]}{n + 3} \mathrm{diag} \left(\lambda_{n+1} - \lambda_{2},\lambda_{n+1} - \lambda_{4},\ldots,\lambda_{n+1} - \lambda_{n-1} \right) \mathbf{v} \mathbf{v}^{\top}
\end{equation}
are the zeros of the polynomial
\begin{equation}\label{eq:4.4e}
G(t) = \prod_{j=1}^{\frac{n-1}{2}}(t - \lambda_{2j}) + \frac{4}{n + 3} \sum_{j=1}^{\frac{n-1}{2}} \sin^{2} \left(\frac{2j \pi}{n + 3} \right) (\lambda_{2j} - \lambda_{n+1}) \prod_{\substack{m=1 \\ m \neq j}}^{\frac{n-1}{2}}(t - \lambda_{2m}).
\end{equation}
Furthermore, if $\lambda_{2},\lambda_{4},\ldots,\lambda_{n-1},\lambda_{n+1}$ are all real and different then the eigenvalues $\psi_{1}, \psi_{2},\ldots, \psi_{\frac{n-1}{2}}$ of \eqref{eq:4.4e} are all simple and
\begin{equation}\label{eq:4.4f}
\mathbf{g}_{k} := \left[
\begin{array}{c}
\frac{(\lambda_{2} - \lambda_{n+1}) \sin\left(\frac{2\pi}{n + 3} \right)}{\lambda_{2} - \psi_{k}} \\
\frac{(\lambda_{4} - \lambda_{n+1}) \sin\left(\frac{4\pi}{n + 3} \right)}{\lambda_{4} - \psi_{k}} \\
\vdots \\
\frac{(\lambda_{n-1} - \lambda_{n+1}) \sin\left[\frac{(n - 1)\pi}{n + 3} \right]}{\lambda_{n-1} - \psi_{k}}
\end{array}
\right]
\end{equation}
is an eigenvector associated to $\psi_{k}$, $k = 1,\ldots,\tfrac{n-1}{2}$.
\end{itemize}
\end{subequations}
\end{lem}

\begin{proof}
Since both assertions can be proven in the same way, we only prove (a). Consider $n \in \mathbb{N}$, $a,b,c,d \in \mathbb{C}$ and $\lambda_{k}$, $k = 1,2,\ldots,n+2$ given by \eqref{eq:2.3}. The proof follows the same steps of the ones displayed in Section~\ref{sec:3}.

\begin{itemize}
\item[i.] Setting
\begin{equation*}
\boldsymbol{\Phi}_{\frac{n}{2}} := \mathrm{diag} \left(\lambda_{3},\lambda_{5},\ldots,\lambda_{n+1} \right) + \frac{4 \sin^{2} \left(\frac{\pi}{n+3} \right)}{n + 3} \mathrm{diag} \left(\lambda_{1} - \lambda_{3},\lambda_{1} - \lambda_{5},\ldots,\lambda_{1} - \lambda_{n+1} \right) \mathbf{u} \mathbf{u}^{\top}
\end{equation*}
with $\mathbf{u}$ given in \eqref{eq:2.4b}, we have
\begin{align*}
\det (t \mathbf{I}_{\frac{n}{2}} - \boldsymbol{\Phi}_{\frac{n}{2}}) &= \det\left[\mathrm{diag} \left(t - \lambda_{3},t - \lambda_{5},\ldots,t - \lambda_{n+1} \right) \right] \cdot \\
&\hspace*{3.6cm} \left[1 - \tfrac{4 \sin^{2} \left(\frac{\pi}{n+3} \right)}{n + 3} \mathbf{u}^{\top}\mathrm{diag} \left(\tfrac{\lambda_{1} - \lambda_{3}}{t - \lambda_{3}},\tfrac{\lambda_{1} - \lambda_{5}}{t - \lambda_{5}},\ldots,\tfrac{\lambda_{1} - \lambda_{n+1}}{t - \lambda_{n+1}} \right) \mathbf{u} \right] \\
&= \prod_{j=1}^{\frac{n}{2}}(t - \lambda_{2j+1}) + \tfrac{4}{n + 3} \sum_{j=1}^{\frac{n}{2}} \sin^{2} \left[\tfrac{(2j + 1) \pi}{n + 3} \right] (\lambda_{2j+1} - \lambda_{1}) \prod_{\substack{m=1 \\ m \neq j}}^{\frac{n}{2}}(t - \lambda_{2m+1}) =: F(t).
\end{align*}
Supposing $\lambda_{1},\lambda_{3},\ldots,\lambda_{n+1}$ all real, distinct and arranged in increasing order by some bijection $\varphi$, i.e. $\lambda_{\varphi(1)} < \lambda_{\varphi(3)} < \ldots < \lambda_{\varphi(n+1)}$, we get
\begin{equation*}
F(t) = \prod_{\substack{j=1 \\ j \neq \vartheta}}^{\frac{n}{2} + 1}\left[t - \lambda_{\varphi(2j-1)} \right] + \tfrac{4}{n + 3} \sum_{\substack{j=1 \\ j \neq \vartheta}}^{\frac{n}{2} + 1} \sin^{2} \left[\tfrac{\varphi(2j - 1) \pi}{n + 3} \right] \left[\lambda_{\varphi(2j-1)} - \lambda_{\varphi(2\vartheta-1)} \right] \prod_{\substack{m=1 \\ m \neq \vartheta,j}}^{\frac{n}{2} + 1}\left[t - \lambda_{\varphi(2m-1)}\right]
\end{equation*}
for some $\vartheta$ satisfying $\varphi(2\vartheta-1) = 1$, $1 \leqslant \vartheta \leqslant \frac{n}{2} + 1$ which yields $F\left[\lambda_{\varphi(2k-1)} \right] F\left[\lambda_{\varphi(2k+1)} \right] < 0$ for all $k=1,\ldots,\frac{n}{2}$. Hence, the zeros $\phi_{1}, \phi_{2},\ldots, \phi_{\frac{n}{2}}$ of $F(t)$ are all simple. From Theorem 5 of \cite{Bunch78} (see \cite{Bunch78}, page $41$), it follows that $\mathbf{f}_{k}$ in \eqref{eq:4.3c} is an eigenvector of $\boldsymbol{\Phi}_{\frac{n}{2}}$ associated to the eigenvalue $\phi_{k}$, $k = 1, \ldots, \frac{n}{2}$.

\item[ii.] Analogously, putting
\begin{equation*}
\boldsymbol{\Psi}_{\frac{n}{2}} := \mathrm{diag} \left(\lambda_{2},\lambda_{4},\ldots,\lambda_{n} \right) + \frac{4 \sin^{2} \left(\frac{\pi}{n+3} \right)}{n + 3} \mathrm{diag} \left(\lambda_{n+2} - \lambda_{2},\lambda_{n+2} - \lambda_{4},\ldots,\lambda_{n+2} - \lambda_{n} \right) \mathbf{v} \mathbf{v}^{\top}
\end{equation*}
with $\mathbf{v}$ given in \eqref{eq:2.4b}, we obtain
\begin{align*}
\det (t \mathbf{I}_{\frac{n}{2}} - \boldsymbol{\Psi}_{\frac{n}{2}}) &= \det\left[\mathrm{diag} \left(t - \lambda_{2},t - \lambda_{4},\ldots,t - \lambda_{n} \right) \right] \cdot \\
&\hspace*{1.9cm} \left[1 - \tfrac{4 \sin^{2} \left(\frac{\pi}{n+3} \right)}{n + 3} \mathbf{v}^{\top}\mathrm{diag} \left(\tfrac{\lambda_{n+2} - \lambda_{2}}{t - \lambda_{2}},\tfrac{\lambda_{n+2} - \lambda_{4}}{t - \lambda_{4}},\ldots,\tfrac{\lambda_{n+2} - \lambda_{n}}{t - \lambda_{n}} \right) \mathbf{v} \right] \\
&= \prod_{j=1}^{\frac{n}{2}}(t - \lambda_{2j}) + \tfrac{4}{n + 3} \sum_{j=1}^{\frac{n}{2}} \sin^{2} \left(\tfrac{2j \pi}{n + 3} \right) (\lambda_{2j} - \lambda_{n+2}) \prod_{\substack{m=1 \\ m \neq j}}^{\frac{n}{2}}(t - \lambda_{2m}) =: G(t).
\end{align*}
Assuming $\lambda_{2},\lambda_{4},\ldots,\lambda_{n}$ all real, distinct and arranged in increasing order by some bijection $\rho$, that is $\lambda_{\rho(2)} < \lambda_{\rho(4)} < \ldots < \lambda_{\rho(n)}$, we have
\begin{equation*}
G(t) = \prod_{\substack{j=1 \\ j \neq \upsilon}}^{\frac{n}{2} + 1}\left[t - \lambda_{\rho(2j)} \right] + \tfrac{4}{n + 3} \sum_{\substack{j=1 \\ j \neq \upsilon}}^{\frac{n}{2} + 1} \sin^{2} \left[\tfrac{\rho(2j) \pi}{n + 3} \right] \left[\lambda_{\rho(2j)} - \lambda_{\rho(2\upsilon)} \right] \prod_{\substack{m=1 \\ m \neq \upsilon,j}}^{\frac{n}{2} + 1}\left[t - \lambda_{\rho(2m)}\right]
\end{equation*}
for some $\upsilon$ such that $\rho(2\upsilon) = n+2$, $1 \leqslant \upsilon \leqslant \frac{n}{2} + 1$ which implies $G\left[\lambda_{\rho(2k)} \right] G\left[\lambda_{\rho(2k+2)} \right] < 0$ for all $k=1,\ldots,\frac{n}{2}$. Thus, the zeros $\psi_{1}, \psi_{2},\ldots, \psi_{\frac{n}{2}}$ of $G(t)$ are all simple. Hence, Theorem 5 of \cite{Bunch78} (see \cite{Bunch78}, page $41$), ensures that $\mathbf{g}_{k}$ in \eqref{eq:4.3f} is an eigenvector of $\boldsymbol{\Psi}_{\frac{n}{2}}$ associated to the eigenvalue $\psi_{k}$, $k = 1, \ldots, \frac{n}{2}$.
\end{itemize}
\end{proof}

Next, we announce a diagonalization for $\mathbf{H}_{n}$ in \eqref{eq:1.1} discarding its eigenvalues.

\begin{thr}\label{thr4}
Let $n \in \mathbb{N}$, $a,b,c,d \in \mathbb{R}$, $\lambda_{k}$, $k = 1,2,\ldots,n+2$ be given by \eqref{eq:2.3} and $\mathbf{H}_{n}$ the $n \times n$ anti-heptadiagonal persymmetric Hankel matrix \eqref{eq:1.1}.

\medskip

\noindent \textnormal{(a)} If $n$ is even, $\lambda_{1},\lambda_{3},\ldots,\lambda_{n+1}$ are all distinct, $\lambda_{2},\lambda_{4},\ldots,\lambda_{n+2}$ are all distinct then
\begin{subequations}
\begin{equation}\label{eq:4.5a}
\mathbf{H}_{n} = \mathbf{R}_{n} \mathbf{P}_{n}^{\top} \left[
\begin{array}{cc}
\mathbf{F}_{\frac{n}{2}} & \mathbf{O} \\
\mathbf{O} & \mathbf{G}_{\frac{n}{2}}
\end{array}
\right] \mathrm{diag} \left(\phi_{1}, \phi_{2},\ldots, \phi_{\frac{n}{2}},\psi_{1},\psi_{2},\ldots,\psi_{\frac{n}{2}} \right) \left[
\begin{array}{cc}
\mathbf{F}_{\frac{n}{2}}^{-1} & \mathbf{O} \\
\mathbf{O} & \mathbf{G}_{\frac{n}{2}}^{-1}
\end{array}
\right] \mathbf{P}_{n} \mathbf{R}_{n}^{-1}
\end{equation}
where $\phi_{1},\phi_{2},\ldots, \phi_{\frac{n}{2}}$ are the zeros of \eqref{eq:4.3b}, $\psi_{1},\psi_{2},\ldots, \psi_{\frac{n}{2}}$ are the zeros of \eqref{eq:4.3e},
\begin{gather}
\mathbf{F}_{\frac{n}{2}} := \left[
\frac{(\lambda_{2k+1} - \lambda_{1}) \sin\left[\frac{(2k+1)\pi}{n + 3} \right]}{\lambda_{2k+1} - \phi_{\ell}} \right]_{k,\ell} \label{eq:4.5b} \\
\mathbf{G}_{\frac{n}{2}} := \left[
\frac{(\lambda_{2k} - \lambda_{n+2}) \sin\left(\frac{2k\pi}{n + 3} \right)}{\lambda_{2k} - \psi_{\ell}}
\right]_{k,\ell} \label{eq:4.5c}
\end{gather}
$\mathbf{R}_{n}$ is the $n \times n$ matrix \eqref{eq:2.4c} and $\mathbf{P}_{n}$ is the $n \times n$ permutation matrix \eqref{eq:2.4d}.
\end{subequations}

\medskip

\noindent \textnormal{(b)} If $n$ is odd, $\lambda_{1},\lambda_{3},\ldots,\lambda_{n},\lambda_{n+2}$ are all distinct, $\lambda_{2},\lambda_{4},\ldots,\lambda_{n-1},\lambda_{n+1}$ are all distinct then
\begin{subequations}
\begin{equation}\label{eq:4.6a}
\mathbf{H}_{n} = \mathbf{R}_{n} \mathbf{P}_{n} \left[
\begin{array}{cc}
\mathbf{F}_{\frac{n+1}{2}} & \mathbf{O} \\
\mathbf{O} & \mathbf{G}_{\frac{n-1}{2}}
\end{array}
\right] \mathrm{diag} \left(\phi_{1}, \phi_{2},\ldots, \phi_{\frac{n+1}{2}},\psi_{1},\psi_{2},\ldots,\psi_{\frac{n-1}{2}} \right) \left[
\begin{array}{cc}
\mathbf{F}_{\frac{n+1}{2}}^{-1} & \mathbf{O} \\
\mathbf{O} & \mathbf{G}_{\frac{n-1}{2}}^{-1}
\end{array}
\right] \mathbf{P}_{n}^{\top} \mathbf{R}_{n}^{-1}
\end{equation}
where $\phi_{1},\phi_{2},\ldots, \phi_{\frac{n+1}{2}}$ are the zeros of \eqref{eq:4.4b}, $\psi_{1},\psi_{2},\ldots, \psi_{\frac{n-1}{2}}$ are the zeros of \eqref{eq:4.4e},
\begin{gather}
\mathbf{F}_{\frac{n+1}{2}} := \left[
\frac{(\lambda_{2k+1} - \lambda_{1}) \sin\left[\frac{(2k+1)\pi}{n + 3} \right]}{\lambda_{2k+1} - \phi_{\ell}} \right]_{k,\ell} \label{eq:4.6b} \\
\mathbf{G}_{\frac{n-1}{2}} := \left[
\frac{(\lambda_{2k} - \lambda_{n+1}) \sin\left(\frac{2k\pi}{n + 3} \right)}{\lambda_{2k} - \psi_{\ell}} \right]_{k,\ell} \label{eq:4.6c}
\end{gather}
$\mathbf{R}_{n}$ is the $n \times n$ matrix \eqref{eq:2.5c} and $\mathbf{P}_{n}$ is the $n \times n$ permutation matrix \eqref{eq:2.5d}.
\end{subequations}
\end{thr}

\begin{proof}
Consider $n \in \mathbb{N}$, $a,b,c,d \in \mathbb{R}$ and $\lambda_{k}$, $k = 1,2,\ldots,n+2$ given by \eqref{eq:2.3}.

\medskip

\noindent (a) According to Lemma~\ref{lem3} and
\begin{equation*}
\mathbf{P}_{n} \mathbf{R}_{n}^{2} \mathbf{P}_{n}^{\top} = \left[
\begin{array}{cc}
  \mathbf{I}_{\frac{n}{2}} - \frac{4 \sin^{2} \left(\frac{\pi}{n+3} \right)}{n+3} \mathbf{u} \mathbf{u}^{\top} & \mathbf{O} \\
  \mathbf{O} & \mathbf{I}_{\frac{n}{2}} - \frac{4 \sin^{2} \left(\frac{\pi}{n+3} \right)}{n+3} \mathbf{v} \mathbf{v}^{\top}
\end{array}
\right]
\end{equation*}
we get
\begin{align*}
\mathbf{H}_{n} &= \mathbf{R}_{n} \mathbf{P}_{n}^{\top} \left[
\begin{array}{cc}
\mathrm{diag} \left(\lambda_{3},\lambda_{5},\ldots,\lambda_{n+1} \right) + \lambda_{1} \mathbf{u} \mathbf{u}^{\top} & \mathbf{O} \\
\mathbf{O} & \mathrm{diag} \left(\lambda_{2},\lambda_{4},\ldots,\lambda_{n} \right) + \lambda_{n+2} \mathbf{v} \mathbf{v}^{\top}
\end{array}
\right] \mathbf{P}_{n} \mathbf{R}_{n}^{2} \mathbf{P}_{n}^{\top} \mathbf{P}_{n} \mathbf{R}_{n}^{-1} \\
&= \mathbf{R}_{n} \mathbf{P}_{n}^{\top} \left[
\begin{array}{c}
\left[\mathrm{diag} \left(\lambda_{3},\lambda_{5},\ldots,\lambda_{n+1} \right) + \lambda_{1} \mathbf{u} \mathbf{u}^{\top} \right] \left[\mathbf{I}_{\frac{n}{2}} - \frac{4 \sin^{2} \left(\frac{\pi}{n+3} \right)}{n+3} \mathbf{u} \mathbf{u}^{\top} \right] \\
\mathbf{O}
\end{array}
\right. \\
& \hspace*{4.5cm} \left.
\begin{array}{c}
\mathbf{O} \\
\left[\mathrm{diag} \left(\lambda_{2},\lambda_{4},\ldots,\lambda_{n} \right) + \lambda_{n+2} \mathbf{v} \mathbf{v}^{\top} \right] \left[\mathbf{I}_{\frac{n}{2}} - \frac{4 \sin^{2} \left(\frac{\pi}{n+3} \right)}{n+3} \mathbf{v} \mathbf{v}^{\top} \right]
\end{array}
\right]
\mathbf{P}_{n} \mathbf{R}_{n}^{-1} \\
&= \mathbf{R}_{n} \mathbf{P}_{n}^{\top} \left[
\begin{array}{c}
\mathrm{diag} \left(\lambda_{3},\lambda_{5},\ldots,\lambda_{n+1} \right) + \frac{4 \sin^{2} \left(\frac{\pi}{n+3} \right)}{n + 3} \mathrm{diag} \left(\lambda_{1} - \lambda_{3},\lambda_{1} - \lambda_{5},\ldots,\lambda_{1} - \lambda_{n+1} \right) \mathbf{u} \mathbf{u}^{\top} \\
\mathbf{O}
\end{array}
\right. \\
& \hspace*{0.7cm} \left.
\begin{array}{c}
\mathbf{O} \\
\mathrm{diag} \left(\lambda_{2},\lambda_{4},\ldots,\lambda_{n} \right) + \frac{4 \sin^{2} \left(\frac{\pi}{n+3} \right)}{n + 3} \mathrm{diag} \left(\lambda_{n+2} - \lambda_{2},\lambda_{n+2} - \lambda_{4},\ldots,\lambda_{n+2} - \lambda_{n} \right) \mathbf{v} \mathbf{v}^{\top}
\end{array}
\right]
\mathbf{P}_{n} \mathbf{R}_{n}^{-1}.
\end{align*}
From Lemma~\ref{lem4},
\begin{equation*}
\mathbf{F}_{\frac{n}{2}} \mathrm{diag}\left(\phi_{1}, \phi_{2},\ldots, \phi_{\frac{n}{2}} \right) \mathbf{F}_{n}^{-1}
\end{equation*}
and
\begin{equation*}
\mathbf{G}_{\frac{n}{2}} \mathrm{diag}\left(\psi_{1}, \psi_{2},\ldots, \psi_{\frac{n}{2}} \right) \mathbf{G}_{n}^{-1}
\end{equation*}
are eigenvalue decompositions for \eqref{eq:4.3a} and \eqref{eq:4.3d}, respectively (see \cite{Ford15}, page $85$). Thus,
\begin{equation*}
\mathbf{H}_{n} = \mathbf{R}_{n} \mathbf{P}_{n}^{\top} \left[
\begin{array}{cc}
\mathbf{F}_{\frac{n}{2}} & \mathbf{O} \\
\mathbf{O} & \mathbf{G}_{\frac{n}{2}}
\end{array}
\right] \mathrm{diag}\left(\phi_{1}, \phi_{2},\ldots, \phi_{\frac{n}{2}},\psi_{1},\psi_{2},\ldots,\psi_{\frac{n}{2}} \right) \left[
\begin{array}{cc}
\mathbf{F}_{\frac{n}{2}}^{-1} & \mathbf{O} \\
\mathbf{O} & \mathbf{G}_{\frac{n}{2}}^{-1}
\end{array}
\right] \mathbf{P}_{n} \mathbf{R}_{n}^{-1}.
\end{equation*}
The proof of (b) is analogous being so omitted.
\end{proof}

Setting
\begin{equation*}
\boldsymbol{\Xi}_{\frac{n}{2}} := \mathrm{diag} \left\{(\lambda_{3} - \lambda_{1}) \sin\left(\frac{3 \pi}{n + 3} \right), (\lambda_{5} - \lambda_{1}) \sin\left(\frac{5 \pi}{n + 3} \right), \ldots, (\lambda_{n+1} - \lambda_{1}) \sin\left[\frac{(n + 1) \pi}{n + 3} \right] \right\}
\end{equation*}
and
\begin{equation*}
\mathbf{Q}_{\frac{n}{2}} := \left[\frac{1}{\lambda_{2k+1} - \phi_{\ell}} \right]_{k,\ell}
\end{equation*}
we have $\mathbf{F}_{\frac{n}{2}} = \boldsymbol{\Xi}_{\frac{n}{2}} \mathbf{Q}_{\frac{n}{2}}$. Straight computations permit us to verify that the inverse of the alternant matrix $\mathbf{Q}_{\frac{n}{2}}$ is
\begin{equation*}
\mathbf{Q}_{\frac{n}{2}}^{-1} = (-1)^{\frac{n}{2}} \left[\frac{\underset{j=1}{\overset{\frac{n}{2}}{\prod}} (\phi_{k} - \lambda_{2j+1}) \underset{\substack{j = 1 \\ j \neq k}}{\overset{\frac{n}{2}}{\prod}} (\phi_{j} - \lambda_{2\ell+1})}{\underset{\substack{j = 1 \\ j \neq k}}{\overset{\frac{n}{2}}{\prod}} (\phi_{k} - \phi_{j}) \underset{\substack{j = 1 \\ j \neq \ell}}{\overset{\frac{n}{2}}{\prod}} (\lambda_{2\ell + 1} - \lambda_{2j + 1})} \right]_{k,\ell}.
\end{equation*}
Hence,
\begin{equation}\label{eq:4.7}
\mathbf{F}_{\frac{n}{2}}^{-1} = \mathbf{Q}_{\frac{n}{2}}^{-1} \boldsymbol{\Xi}_{\frac{n}{2}}^{-1} = (-1)^{\frac{n}{2}} \left[\frac{\underset{j=1}{\overset{\frac{n}{2}}{\prod}} (\phi_{k} - \lambda_{2j+1}) \underset{\substack{j = 1 \\ j \neq k}}{\overset{\frac{n}{2}}{\prod}} (\phi_{j} - \lambda_{2\ell+1})}{(\lambda_{2\ell + 1} - \lambda_{1}) \sin \left[\frac{(2\ell+1)\pi}{n + 3} \right] \underset{\substack{j = 1 \\ j \neq k}}{\overset{\frac{n}{2}}{\prod}} (\phi_{k} - \phi_{j}) \underset{\substack{j = 1 \\ j \neq \ell}}{\overset{\frac{n}{2}}{\prod}} (\lambda_{2\ell + 1} - \lambda_{2j + 1})} \right]_{k,\ell}
\end{equation}
Similarly,
\begin{gather}
\mathbf{G}_{\frac{n}{2}}^{-1} = (-1)^{\frac{n}{2}} \left[\frac{\underset{j=1}{\overset{\frac{n}{2}}{\prod}} (\psi_{k} - \lambda_{2j}) \underset{\substack{j = 1 \\ j \neq k}}{\overset{\frac{n}{2}}{\prod}} (\psi_{j} - \lambda_{2\ell})}{(\lambda_{2\ell} - \lambda_{n+2}) \sin \left(\frac{2\ell\pi}{n + 3} \right) \underset{\substack{j = 1 \\ j \neq k}}{\overset{\frac{n}{2}}{\prod}} (\psi_{k} - \psi_{j}) \underset{\substack{j = 1 \\ j \neq \ell}}{\overset{\frac{n}{2}}{\prod}} (\lambda_{2\ell} - \lambda_{2j})} \right]_{k,\ell} \label{eq:4.8} \\
\mathbf{F}_{\frac{n+1}{2}}^{-1} = (-1)^{\frac{n+1}{2}} \left[\frac{\underset{j=1}{\overset{\frac{n+1}{2}}{\prod}} (\phi_{k} - \lambda_{2j+1}) \underset{\substack{j = 1 \\ j \neq k}}{\overset{\frac{n+1}{2}}{\prod}} (\phi_{j} - \lambda_{2\ell+1})}{(\lambda_{2\ell + 1} - \lambda_{1}) \sin \left[\frac{(2\ell+1)\pi}{n + 3} \right] \underset{\substack{j = 1 \\ j \neq k}}{\overset{\frac{n+1}{2}}{\prod}} (\phi_{k} - \phi_{j}) \underset{\substack{j = 1 \\ j \neq \ell}}{\overset{\frac{n+1}{2}}{\prod}} (\lambda_{2\ell + 1} - \lambda_{2j + 1})} \right]_{k,\ell} \label{eq:4.9} \\
\mathbf{G}_{\frac{n-1}{2}}^{-1} = (-1)^{\frac{n-1}{2}} \left[\frac{\underset{j=1}{\overset{\frac{n-1}{2}}{\prod}} (\psi_{k} - \lambda_{2j}) \underset{\substack{j = 1 \\ j \neq k}}{\overset{\frac{n-1}{2}}{\prod}} (\psi_{j} - \lambda_{2\ell})}{(\lambda_{2\ell} - \lambda_{n+1}) \sin \left(\frac{2\ell\pi}{n + 3} \right) \underset{\substack{j = 1 \\ j \neq k}}{\overset{\frac{n-1}{2}}{\prod}} (\psi_{k} - \psi_{j}) \underset{\substack{j = 1 \\ j \neq \ell}}{\overset{\frac{n-1}{2}}{\prod}} (\lambda_{2\ell} - \lambda_{2j})} \right]_{k,\ell}. \label{eq:4.10}
\end{gather}

The ultimate statement below gives us a formula to compute integer powers of $\mathbf{H}_{n}$ in \eqref{eq:1.1} only at the expense of the zeros of $F(t)$ and $G(t)$ presented in Lemma~\ref{lem4}.

\begin{cor}
Let $n,m \in \mathbb{N}$, $a,b,c,d \in \mathbb{R}$, $\lambda_{k}$, $k = 1,2,\ldots,n+2$ be given by \eqref{eq:2.3} and $\mathbf{H}_{n}$ the $n \times n$ anti-heptadiagonal persymmetric Hankel matrix \eqref{eq:1.1}.

\medskip

\noindent \textnormal{(a)} If $n$ is even, $\lambda_{1},\lambda_{3},\ldots,\lambda_{n+1}$ are all distinct, $\lambda_{2},\lambda_{4},\ldots,\lambda_{n+2}$ are all distinct then
\begin{equation}\label{eq:4.11}
\begin{split}
\mathbf{H}_{n}^{m} = \mathbf{R}_{n} \mathbf{P}_{n}^{\top} \left[
\begin{array}{cc}
\mathbf{F}_{\frac{n}{2}} & \mathbf{O} \\
\mathbf{O} & \mathbf{G}_{\frac{n}{2}}
\end{array}
\right] \mathrm{diag} & \left(\phi_{1}^{m}, \phi_{2}^{m},\ldots, \phi_{\frac{n}{2}}^{m},\psi_{1}^{m},\psi_{2}^{m},\ldots,\psi_{\frac{n}{2}}^{m} \right) \\
& \left[
\begin{array}{cc}
\mathbf{F}_{\frac{n}{2}}^{-1} & \mathbf{O} \\
\mathbf{O} & \mathbf{G}_{\frac{n}{2}}^{-1}
\end{array}
\right] \left[
\begin{array}{cc}
  \mathbf{I}_{\frac{n}{2}} + \mathbf{u} \mathbf{u}^{\top} & \mathbf{O} \\
  \mathbf{O} & \mathbf{I}_{\frac{n}{2}} + \mathbf{v} \mathbf{v}^{\top}
\end{array}
\right] \mathbf{P}_{n} \mathbf{R}_{n}
\end{split}
\end{equation}
where $\phi_{1},\phi_{2},\ldots, \phi_{\frac{n}{2}}$ are the zeros of \eqref{eq:4.3a}, $\psi_{1},\psi_{2},\ldots, \psi_{\frac{n}{2}}$ are the zeros of \eqref{eq:4.3e}, $\mathbf{F}_{\frac{n}{2}}$ is the $\frac{n}{2} \times \frac{n}{2}$ matrix \eqref{eq:4.5b}, $\mathbf{G}_{\frac{n}{2}}$ is the $\frac{n}{2} \times \frac{n}{2}$ matrix \eqref{eq:4.5c}, $\mathbf{F}_{\frac{n}{2}}^{-1}$ is the $\frac{n}{2} \times \frac{n}{2}$ matrix \eqref{eq:4.7}, $\mathbf{G}_{\frac{n}{2}}^{-1}$ is the $\frac{n}{2} \times \frac{n}{2}$ matrix \eqref{eq:4.8}, $\mathbf{R}_{n}$ is the $n \times n$ matrix \eqref{eq:2.4c}, $\mathbf{P}_{n}$ is the $n \times n$ permutation matrix \eqref{eq:2.4d} and $\mathbf{u}, \mathbf{v}$ are given by \eqref{eq:2.4b}. Moreover, if $\mathbf{H}_{n}$ is nonsingular then \eqref{eq:4.11} holds for any integer $m$.

\medskip

\noindent \textnormal{(b)} If $n$ is odd, $\lambda_{1},\lambda_{3},\ldots,\lambda_{n},\lambda_{n+2}$ are all distinct, $\lambda_{2},\lambda_{4},\ldots,\lambda_{n-1},\lambda_{n+1}$ are all distinct then
\begin{equation}\label{eq:4.12}
\begin{split}
\mathbf{H}_{n}^{m} = \mathbf{R}_{n} \mathbf{P}_{n} \left[
\begin{array}{cc}
\mathbf{F}_{\frac{n+1}{2}} & \mathbf{O} \\
\mathbf{O} & \mathbf{G}_{\frac{n-1}{2}}
\end{array}
\right] & \mathrm{diag} \left(\phi_{1}^{m}, \phi_{2}^{m},\ldots, \phi_{\frac{n+1}{2}}^{m},\psi_{1}^{m},\psi_{2}^{m},\ldots,\psi_{\frac{n-1}{2}}^{m} \right) \\
& \left[
\begin{array}{cc}
\mathbf{F}_{\frac{n+1}{2}}^{-1} & \mathbf{O} \\
\mathbf{O} & \mathbf{G}_{\frac{n-1}{2}}^{-1}
\end{array}
\right] \left[
\begin{array}{cc}
  \mathbf{I}_{\frac{n+1}{2}} + \mathbf{u} \mathbf{u}^{\top} & \mathbf{O} \\
  \mathbf{O} & \mathbf{I}_{\frac{n-1}{2}} + \mathbf{v} \mathbf{v}^{\top}
\end{array}
\right] \mathbf{P}_{n}^{\top} \mathbf{R}_{n}^{\top}
\end{split}
\end{equation}
where $\phi_{1},\phi_{2},\ldots, \phi_{\frac{n+1}{2}}$ are the zeros of \eqref{eq:4.3a}, $\psi_{1},\psi_{2},\ldots, \psi_{\frac{n-1}{2}}$ are the zeros of \eqref{eq:4.3e},  $\mathbf{F}_{\frac{n+1}{2}}$ is the $\frac{n+1}{2} \times \frac{n+1}{2}$ matrix \eqref{eq:4.6b}, $\mathbf{G}_{\frac{n-1}{2}}$ is the $\frac{n-1}{2} \times \frac{n-1}{2}$ matrix \eqref{eq:4.6c}, $\mathbf{F}_{\frac{n+1}{2}}^{-1}$ is the $\frac{n+1}{2} \times \frac{n+1}{2}$ matrix \eqref{eq:4.9}, $\mathbf{G}_{\frac{n-1}{2}}^{-1}$ is the $\frac{n-1}{2} \times \frac{n-1}{2}$ matrix \eqref{eq:4.10}, $\mathbf{R}_{n}$ is the $n \times n$ matrix \eqref{eq:2.5c}, $\mathbf{P}_{n}$ is the $n \times n$ permutation matrix \eqref{eq:2.5d} and $\mathbf{u}, \mathbf{v}$ are defined in \eqref{eq:2.5b}. Furthermore, if $\mathbf{H}_{n}$ is nonsingular then \eqref{eq:4.12} is valid for every integer $m$.
\end{cor}

\begin{proof}
According to Theorem~\ref{thr4}, it suffices to observe that for $n$ even
\begin{equation*}
\mathbf{P}_{n} \mathbf{R}_{n}^{-1} = \left[
\begin{array}{cc}
  \mathbf{I}_{\frac{n}{2}} + \mathbf{u} \mathbf{u}^{\top} & \mathbf{O} \\
  \mathbf{O} & \mathbf{I}_{\frac{n}{2}} + \mathbf{v} \mathbf{v}^{\top}
\end{array}
\right] \mathbf{P}_{n} \mathbf{R}_{n},
\end{equation*}
and for $n$ odd,
\begin{equation*}
\mathbf{P}_{n}^{\top} \mathbf{R}_{n}^{-1} = \left[
\begin{array}{cc}
  \mathbf{I}_{\frac{n+1}{2}} + \mathbf{u} \mathbf{u}^{\top} & \mathbf{O} \\
  \mathbf{O} & \mathbf{I}_{\frac{n-1}{2}} + \mathbf{v} \mathbf{v}^{\top}
\end{array}
\right] \mathbf{P}_{n}^{\top} \mathbf{R}_{n}^{\top}.
\end{equation*}
\end{proof}

\begin{rem}
It is worth pointing out that all results presented here for matrices $\mathbf{H}_{n}$ having the form \eqref{eq:1.1} are still valid for anti-pendiagonal persymmetric Hankel matrices or anti-tridiagonal persymmetric Hankel matrices, that is to say, for matrices $\mathbf{H}_{n}$ in \eqref{eq:1.1} such that $a = 0$ or $a = b = 0$.
\end{rem}

\section*{Acknowledgements}

This work is a contribution to the Project UID/GEO/04035/2013, funded by FCT - Funda\c{c}\~{a}o para a Ci\^{e}ncia e a Tecnologia, Portugal.


\begin{thebibliography}{00}

\bibitem{Bini83} D. Bini and M. Capovani, Spectral and computational properties of band symmetric Toeplitz matrices, Linear Algebra Appl. 52/53 (1983) 99--126.

\bibitem{Bultheel97} A. Bultheel and M. Van Barel, Linear Algebra: Rational Approximation and Orthogonal Polynomials, Studies in Computational Mathematics 6, North-Holland, 1997.

\bibitem{Bunch78} J.R. Bunch, C.P. Nielsen, D.C. Sorensen, Rank-one modification of the symmetric eigenproblem, Numer. Math. 31 (1978) 31--48.

\bibitem{Datta88} B.N. Datta, C.R. Johnson, M.A. Kaashoek, R.J. Plemmons, E.D. Sontag, Linear Algebra in Signals, Systems, and Control, SIAM, 1988.

\bibitem{Fasino88} D. Fasino, Spectral and structural properties of some pentadiagonal symmetric matrices, Calcolo 25(4) (1988) 301--310.

\bibitem{Ford15} W. Ford, Numerical Linear Algebra with Applications Using MATLAB, Academic Press, 2015.

\bibitem{Gutierrez11} J. Guti\'{e}rrez-Guti\'{e}rrez, Powers of complex persymmetric or skew-persymmetric anti-tridiagonal matrices with constant anti-diagonals, Appl. Math. Comput. 217 (2011) 6125--6132.

\bibitem{Gutierrez14} J. Guti\'{e}rrez-Guti\'{e}rrez, Eigenvalue decomposition for persymmetric Hankel matrices with at most three non-zero anti-diagonals, Appl. Math. Comput. 234 (2014) 333--338.

\bibitem{Harville97} D.A. Harville, Matrix Algebra From a Statistician's Perspective, Springer-Verlag New York, 1997.

\bibitem{Horn13} R.A. Horn and C.R. Johnson, Matrix Analysis, 2nd ed., Cambridge University Press, 2013.

\bibitem{Lita16} J. Lita da Silva, On anti-pentadiagonal persymmetric Hankel matrices with perturbed corners, Comput. Math. Appl. 72 (2016), 415--426.

\bibitem{Kailath80} T. Kailath, Linear Systems, Prentice-Hall, 1980.

\bibitem{Miller81} K.S. Miller, On the inverse of the sum of matrices, Math. Mag. 54(2) (1981) 67--72.

\bibitem{Olshevsky01} V. Olshevsky, M. Stewart, Stable factorization for Hankel and Hankel-like matrices, Numer. Linear Algebra Appl. 8 (2001) 401--434.

\bibitem{Pissanetsky84} Pissanetsky S. Sparse Matrix Technology. London: Academic Press, 1984.

\bibitem{Rimas13a} J. Rimas, Integer powers of real odd order skew-persymmetric anti-tridiagonal matrices with constant anti-diagonals ($\mathrm{antitridiag}_{n}(a,c,-a)$, $a \in R \setminus \{0 \}$, $c \in R$), Appl. Math. Comput. 219 (2013) 7075--7088.

\bibitem{Rimas13b} J. Rimas, Integer powers of real even order anti-tridiagonal Hankel matrices of the form $\mathrm{antitridiag}_{n}(a,c,-a)$, Appl. Math. Comput. 225 (2013) 204--215.

\bibitem{Williams14} G. Williams, Linear Algebra with Applications, 8th ed., Jones \& Bartlett Learning, Burlington, MA, 2014.

\bibitem{Wu10} Honglin Wu, On computing of arbitrary positive powers for one type of anti-tridiagonal matrices of even order, Appl. Math. Comput. 217 (2010) 2750--2756.

\bibitem{Yin08} Qingxiang Yin, On computing of arbitrary positive powers for anti-tridiagonal matrices of even order, Appl. Math. Comput. 203 (2008) 252--257.
\end{thebibliography}
\end{document}